\documentclass[11pt]{amsart}

\usepackage{a4wide,amsmath,amssymb,dsfont,epsfig,graphicx,subfig,verbatim,multicol,enumitem,lipsum}
\usepackage[bottom]{footmisc}

\newtheorem{theorem}{Theorem}
\newtheorem{proposition}[theorem]{Proposition}

\newtheorem{corollary}[theorem]{Corollary}

\newtheorem*{algorithma}{Algorithm A}
\newtheorem*{algorithmb}{Algorithm B}

\newtheorem*{theoremX}{Theorem}

\newtheorem*{stc}{Strong Terminating Conjecture \cite{SchultzShiflett}}

\newtheorem*{rl}{Reproduction Lemma}

\renewenvironment{proof}[1][Proof:]{\noindent\textbf{#1\\} }{\ \rule{0.5em}{0.5em}}
\newenvironment{proofoftheoremI}[1][Proof of part i):]{\noindent\textbf{#1\\} }{\ \rule{0.5em}{0.5em}}
\newenvironment{proofoftheoremII}[1][Proof of part ii):]{\noindent\textbf{#1\\} }{\ \rule{0.5em}{0.5em}}

\newcommand{\cM}{\mathcal{M}}

\newcommand{\mmm}{{\sc mmm}}
\newcommand{\AP}{\mathrm{AP}}
\newcommand{\PP}{\mathrm{P}}

\usepackage{pgf,tikz}
\usepackage{tkz-fct}
\usepackage{pgfplots}
\usetikzlibrary{arrows,trees}
\pgfplotsset{ every non boxed x axis/.append style={x axis line style=<->},
     every non boxed y axis/.append style={y axis line style=<->}, every axis/.append style={font=\tiny}}
\definecolor{newpurple}{RGB}{195, 22, 140}
\definecolor{newgray}{RGB}{240, 240, 240}
\definecolor{newlightblue}{RGB}{0, 175, 158}
\definecolor{newblue}{RGB}{47, 50, 145}
\definecolor{newyellow}{RGB}{232, 222, 0}
\definecolor{newgreen}{RGB}{0, 155, 1}
\begin{document}
\title[]
{On the unboundedness of the transit time of mean-median orbits}
\author{Jonathan Hoseana and Franco Vivaldi}
\address{School of Mathematical Sciences, Queen Mary,
University of London,
London E1 4NS, UK}
%\email{f.vivaldi@maths.qmul.ac.uk}
%\urladdr{http://www.maths.qmul.ac.uk/\~{}fv}

\begin{abstract}
The transit time of mean-median orbits ---the time it takes for an 
orbit to become stationary--- has been conjectured to be finite
\cite{SchultzShiflett} but unbounded \cite{HoseanaMSc} 
over the rationals.
Through a study of some near-regular structures in these orbits,
we construct two non-trivial sequences of initial 
sets of increasing size for which the transit time grows linearly 
and quadratically, respectively, with the size of the set.
\end{abstract}
\date{\today}

\maketitle

%-------------------------------------------------------------------
\section{Introduction}\label{section:Introduction}

The \textit{mean-median map} (\mmm) is a dynamical system over the space of finite multisets\footnote{Hereafter referred to simply as sets.} of real numbers. 
This map enlarges a set by adjoining to it a new real number so that the arithmetic mean of the enlarged set equals the median of the original set. 
The iteration of this deceptively simple map produces a novel and intriguing 
dynamics, where the orbits are shaped by the interaction with previous iterates.
This dynamics is poorly understood, in spite of considerable research \cite{SchultzShiflett,ChamberlandMartelli,HoseanaMSc,CellarosiMunday,HoseanaVivaldi}. Several conjectures ---all still open--- have been formulated, the most prominent being:\bigskip

\begin{stc}
For every initial set, the sequence of new numbers generated by iterating the \mmm\ is eventually constant.
\end{stc}\bigskip

We believe that this conjecture holds over the rationals, although not necessarily over larger fields \cite{HoseanaMSc}. Accordingly, we associate to any initial set $\xi$ the time step $\tau(\xi)\in\mathbb{N}_{>|\xi|}\cup\{\infty\}$ at which its \mmm\ sequence becomes constant, and the limit $m(\xi)\in\mathbb{R}$ of this sequence, if it exists. These numbers are called the \textit{transit time} and the \textit{limit} of $\xi$, respectively.

%%%%%%%%%%%%%%%%%%%%%%%%%%%%%%%%%%%%%%%%%%%%%%%%%%%%%%%% FIGURE
\begin{figure}[t!]
\centering
\input{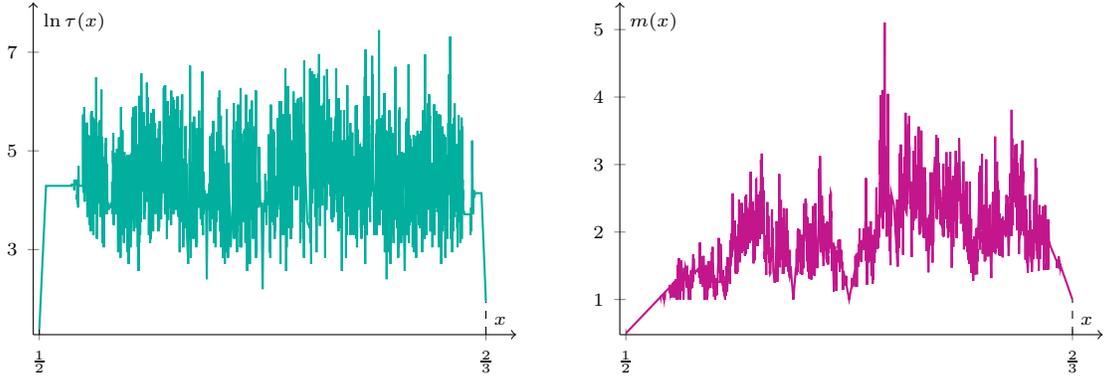}
\caption{\label{fig:mtau}\small
The transit time (left, in log-scale) and limit (right) of the initial set $[0,x,1]$ as functions of $x\in\left[\frac{1}{2},\frac{2}{3}\right]$ sampled over fractions with denominator up to $2000$.
}
\end{figure}
%%%%%%%%%%%%%%%%%%%%%%%%%%%%%%%%%%%%%%%%%%%%%%%%%%%%%%%%

In the simplest non-trivial case ---a three-element initial set--- it can be shown that the dynamics is conjugate to that of the initial set $[0,x,1]$, $x\in\left[\frac{1}{2},\frac{2}{3}\right]$, so that the transit time and limit are functions of $x$ (figure \ref{fig:mtau}). In this case, computational evidence [figure \ref{fig:tauoriginal} (a)] suggests that, once the very large fluctuations are averaged out, the transit time for $x=\frac{p}{q}$ depends algebraically on the denominator $q$ with exponent $\alpha\approx 0.42$ \cite{HoseanaMSc,HoseanaVivaldi}. 
Furthermore, the maximum transit time of all fractions with denominator at 
most $q$ is observed to depend algebraically on $q$ with the larger
exponent $\beta\approx 1.45$ [figure \ref{fig:tauoriginal} (b)]. 
In particular, these observations suggest that the transit time function 
over the space of three-element initial sets is unbounded.

%%%%%%%%%%%%%%%%%%%%%%%%%%%%%%%%%%%%%%%%%%%%%%%%%%%%%%%% FIGURE
\begin{figure}[b!]
\centering
\input{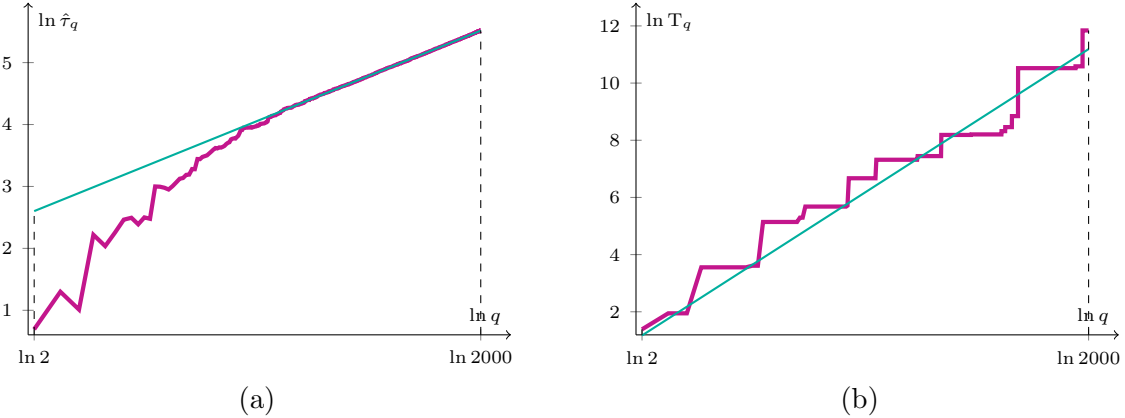}
\caption{\label{fig:tauoriginal}\small
(a) Logarithm of the $q$-th Cesaro mean $\hat\tau_q$ of the sequence $\left(\tau_q\right)_{q=2}^{2000}$ of average transit times $\tau_q$ of fractions with denominator $q$ in the interval $\left[\frac{1}{2},\frac{2}{3}\right]$, versus $\ln q$. The line has equation $\ln\hat\tau_q=\alpha \ln q + 2.31$, where $\alpha\approx 0.42$. (b) Logarithm of the maximum transit time $\text{T}_q$ of fractions with denominator up to $q$ in the interval $\left[\frac{1}{2},\frac{2}{3}\right]$, versus $\ln q$, for $q\in\{2,\ldots,2000\}$. The line has equation $\ln\text{T}_q=\beta \ln q + 0.17$, where $\beta\approx 1.45$.
}
\end{figure}
%%%%%%%%%%%%%%%%%%%%%%%%%%%%%%%%%%%%%%%%%%%%%%%%%%%%%%%%

We wish to study the unboundedness of transit time of \mmm\ orbits in a more controlled setting. 
For this purpose, we consider specific families of larger initial sets whose dynamics is more predictable. 
Trivially, any initial set $\xi$ whose mean and median coincide has 
$\tau(\xi)=|\xi|+1$, the orbit becoming constant after just one 
iteration. Furthermore, as we shall see shortly, for every even 
$k\in\mathbb{N}$ there exists a sequence of sets $\xi$ with
$|\xi|\to\infty$ and $\tau(\xi)=|\xi|+k$. 

In this paper, we shall improve upon such straightforward 
$\tau(\xi)\sim|\xi|$ asymptotics by establishing the following 
result:\bigskip

\begin{samepage}
\begin{theoremX}\textcolor{white}{a}
\begin{enumerate}
\item[i)] For every $k\in\mathbb{N}_0$ there exists a sequence of 
sets $\xi$ of increasing size for which 
$$\tau(\xi)\sim \sqrt{\frac{3k+5}{k+1}}\,|\xi|.$$
\item[ii)] There exists a sequence of sets $\xi$ of increasing size 
for which 
$$\tau(\xi)\sim \frac{|\xi|^2}{2}.$$
\end{enumerate}
\end{theoremX}
\end{samepage}\bigskip

\noindent 
The proof will require the development of a theory of certain 
near-regular structures that are found in some \mmm\ sequences 
(figure \ref{fig:205}).

To relate this result to the conjectured behaviour 
of the system $[0,x,1]$ 
we shall introduce a suitable measure of complexity of rational sets. 
We will then show that there is a subsequence of ii) above whose
transit time grows algebraically with the complexity,
with exponent lying between
the empirical exponents $\alpha$ and $\beta$ of
figure \ref{fig:tauoriginal}.

%%%%%%%%%%%%%%%%%%%%%%%%%%%%%%%%%%%%%%%%%%%%%%%%%%%%%%%% FIGURE
\begin{figure}
\centering
\input{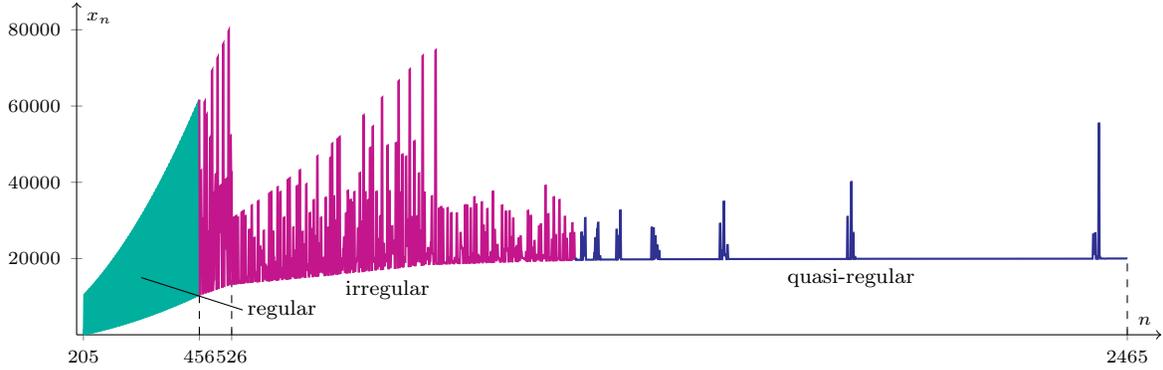}
\caption{\label{fig:205}\small The normal form orbit of order $205$ which features a regular, irregular, and quasi-regular structure.}
\end{figure}
%%%%%%%%%%%%%%%%%%%%%%%%%%%%%%%%%%%%%%%%%%%%%%%%%%%%%%%%

We now define the \mmm\ precisely, and survey previous works.
Let $\xi_{n_0}=\left[x_1,\ldots,x_{n_0}\right]$, $n_0\in\mathbb{N}$, be an initial set. For every $n\geqslant n_0$, the image $\xi_{n+1}$ of the set $\xi_n$ is obtained by adjoining the number $x_{n+1}$ for which
\begin{equation}\label{eq:defofMMM}
\left\langle\xi_{n+1}\right\rangle=\mathcal{M}\left(\xi_n\right),\qquad\text{i.e.},\qquad x_{n+1}=(n+1)\mathcal{M}\left(\xi_n\right)-\mathcal{S}\left(\xi_n\right),
\end{equation}
where $\left\langle \xi_n\right\rangle$, $\mathcal{M}\left(\xi_n\right)$, and $\mathcal{S}\left(\xi_n\right)$ denote the mean, median\footnote{The central element of the ordered list $\xi_n$ if $n$ is odd, otherwise the mean of the two central elements.}, and sum of elements of $\xi_n$, respectively. Formally, we write
$$\xi_{n+1}:=\xi_n\uplus\left[x_{n+1}\right],$$
where the union operator $\uplus$ increases the multiplicity of the number $x_{n+1}$ in $\xi_n$ by $1$ to obtain $\xi_{n+1}$. The iterations thus produce a sequence of sets $\left(\xi_n\right)_{n=n_0}^\infty$, an \textit{orbit} $\left(x_n\right)_{n=1}^\infty$, and a \textit{median sequence} $\left(\cM_n\right)_{n=n_0}^\infty$, where $\cM_n:=\cM\left(\xi_n\right)$. If the sequence is eventually constant, it is said to \textit{stabilise}.

The recursion \eqref{eq:defofMMM} can be rewritten as a second-order recursion involving only the medians, namely
\begin{equation}\label{eq:MMMmedian}
x_{n+1}=(n+1)\mathcal{M}_n-n\mathcal{M}_{n-1}=n\Delta\mathcal{M}_n+\mathcal{M}_n,\quad\text{where}\quad \Delta\mathcal{M}_n:=\mathcal{M}_n-\mathcal{M}_{n-1},
\end{equation}
which is valid for every $n\geqslant n_0+1$.
From this it follows that two equal consecutive medians imply stabilisation 
\cite[page 197]{SchultzShiflett} and that the median sequence is monotonic \cite[theorem 2.1]{ChamberlandMartelli}. Due to symmetry ---replacing $\xi_{n_0}$ by $-\xi_{n_0}$ reverses the monotonicity--- we only need to consider non-decreasing median sequences.

From equation \eqref{eq:MMMmedian} we see that every iterate beyond 
the second one depends only on two central elements of the evolving set. 
In addition, if the initial set has zero median, then
equation \eqref{eq:defofMMM} implies that the first iterate depends 
only on the sum of its elements. Exploiting these facts, it is possible to construct an odd-sized initial set of the form\footnote{The elements of $\xi_{n_0}$ are written in non-decreasing order.}
\begin{equation}\label{eq:NFset}
\xi_{n_0}=\left[\ell_1,\ldots,\ell_{\frac{n_0-1}{2}},0,1,u_1,\ldots,u_{\frac{n_0-3}{2}}\right],
\end{equation}
where the $\ell_i$s and $u_i$s are chosen to be large enough in magnitude while keeping their sum fixed, so that these elements do not participate in the dynamics of $\xi_{n_0}$ for any time interval of interest. In particular, if the $\ell_i$s, the $u_i$s, and $x_{n_0+1}$ are infinitely far, then the sequence $\left(x_n\right)_{n=t}^{\infty}$, where $t:=n_0+2$, is the so-called \textit{normal form orbit of order $t$} \cite[section 6]{HoseanaVivaldi} which is conjugate to the orbit near any stabilised local minimum of the limit function having transit time $t$. Such minima have an intricate structure, see figure \ref{fig:mtau} (right).

Due to the absence of non-central elements in the initial set, the normal form orbit of order $t$ features an initial time interval of length $N_t\sim t\sqrt{5}$ in which the orbit is predictable and explicitly computable;
this is the so-called \textit{regular phase} 
(figure \ref{fig:205}) \cite[lemma 13]{HoseanaVivaldi}. If one of these non-central elements is repositioned precisely at an 
odd-indexed median in the regular phase, then the whole orbit 
becomes predictable; it consists of a regular structure which is truncated 
due to stabilisation when this median is reached. For any given even $k\in\mathbb{N}$ we can use this idea ---choosing sufficiently long regular phases--- to construct initial sets $\xi$ of arbitrarily large size for which 
$\tau(\xi)=|\xi|+k$, as suggested earlier.

%%%%%%%%%%%%%%%%%%%%%%%%%%%%%%%%%%%%%%%%%%%%%%%%%%%%%%%% FIGURE
\begin{figure}
\centering
\begin{tikzpicture}
\begin{axis}[
	xmin=10.82432432,
	xmax=24.87837838,
	ymin=0,
	ymax=121.9090909,
    ytick={5.5000, 6.5000, 13, 14, 30.250, 34.750, 61.750, 68.250},
    yticklabels={},
    xtick={11,15,19,23},
    xticklabels={$t$,$t+4$,$t+8$,$t+12$},
	axis lines=middle,
	samples=100,
	xlabel=$n$,
	ylabel={$x_n$,\,$\mathcal{M}_n$},
	width=8cm,
	height=6cm,
	clip=false,
	axis lines=middle,
    x axis line style=->,
    y axis line style=->,
]
\draw[thick,color=newpurple] plot coordinates {(axis cs:11., 5.5000) (axis cs:12., 6.5000) (axis cs:13., 30.250) (axis cs:14., 34.750) (axis cs:15., 13.) (axis cs:16., 14.) (axis cs:17., 61.750) (axis cs:18., 68.250) (axis cs:19., 22.500) (axis cs:20., 23.500) (axis cs:21., 103.25) (axis cs:22., 111.75) (axis cs:23., 34.)};
\draw[thick,color=newpurple,dashed] (axis cs:23., 34.)--(axis cs:24., 35.);

\draw[thick,color=newblue] plot coordinates {(axis cs:11., 1.) (axis cs:12., 3.2500) (axis cs:13., 5.5000) (axis cs:14., 6.) (axis cs:15., 6.5000) (axis cs:16., 9.7500) (axis cs:17., 13.) (axis cs:18., 13.500) (axis cs:19., 14.) (axis cs:20., 18.250) (axis cs:21., 22.500) (axis cs:22., 23.) (axis cs:23., 23.500)};
\draw[thick,color=newblue,dashed] (axis cs:23., 23.500)--(axis cs:24., 26.875);

\draw[dashed] (axis cs:10.82432432, 5.5000) -- (axis cs:13, 5.5000);
\draw[dashed] (axis cs:10.82432432, 6.5000) -- (axis cs:15, 6.5000);
\draw[dashed] (axis cs:10.82432432, 13) -- (axis cs:17, 13);
\draw[dashed] (axis cs:10.82432432, 14) -- (axis cs:19, 14);

\draw[dashed] (axis cs:10.82432432, 30.250) -- (axis cs:13, 30.250);
\draw[dashed] (axis cs:10.82432432, 34.750) -- (axis cs:14, 34.750);
\draw[dashed] (axis cs:10.82432432, 61.750) -- (axis cs:17, 61.750);
\draw[dashed] (axis cs:10.82432432, 68.250) -- (axis cs:18, 68.250);

\draw[dashed] (axis cs:11., 0) -- (axis cs:11., 5.5000);
\draw[dashed] (axis cs:15., 0) -- (axis cs:15., 13);
\draw[dashed] (axis cs:19., 0) -- (axis cs:19., 22.5000);
\draw[dashed] (axis cs:23., 0) -- (axis cs:23., 34);

\draw[dashed] (axis cs:19., 22.500) -- (axis cs:21., 22.500);
\draw[dashed] (axis cs:20., 23.500) -- (axis cs:23., 23.500);

\fill[newpurple] (axis cs:11., 5.5000) circle (1.5pt);
\fill[newpurple] (axis cs:12., 6.5000) circle (1.5pt);
\fill[newpurple] (axis cs:13., 30.250) circle (1.5pt);
\fill[newpurple] (axis cs:14., 34.750) circle (1.5pt);
\fill[newpurple] (axis cs:15., 13.) circle (1.5pt);
\fill[newpurple] (axis cs:16., 14.) circle (1.5pt);
\fill[newpurple] (axis cs:17., 61.750) circle (1.5pt);
\fill[newpurple] (axis cs:18., 68.250) circle (1.5pt);
\fill[newpurple] (axis cs:19., 22.500) circle (1.5pt);
\fill[newpurple] (axis cs:20., 23.500) circle (1.5pt);
\fill[newpurple] (axis cs:21., 103.25) circle (1.5pt);
\fill[newpurple] (axis cs:22., 111.75) circle (1.5pt);
\fill[newpurple] (axis cs:23., 34.) circle (1.5pt);

\fill[newblue] (axis cs:11., 1.) circle (1.5pt);
\fill[newblue] (axis cs:12., 3.2500) circle (1.5pt);
\fill[newblue] (axis cs:13., 5.5000) circle (1.5pt);
\fill[newblue] (axis cs:14., 6.) circle (1.5pt);
\fill[newblue] (axis cs:15., 6.5000) circle (1.5pt);
\fill[newblue] (axis cs:16., 9.7500) circle (1.5pt);
\fill[newblue] (axis cs:17., 13.) circle (1.5pt);
\fill[newblue] (axis cs:18., 13.500) circle (1.5pt);
\fill[newblue] (axis cs:19., 14.) circle (1.5pt);
\fill[newblue] (axis cs:20., 18.250) circle (1.5pt);
\fill[newblue] (axis cs:21., 22.500) circle (1.5pt);
\fill[newblue] (axis cs:22., 23.) circle (1.5pt);
\fill[newblue] (axis cs:23., 23.500) circle (1.5pt);

\node[below right,yshift=2pt] at (axis cs:23,34) {\tiny linear iterates};
\node[above] at (axis cs:22., 111.75) {\tiny quadratic iterates};

\node[left] at (axis cs:9,5) {\tiny $x_t=\frac{t}{2}$};
\node[left] at (axis cs:9,17.5) {\tiny $x_{t+1}=\frac{t}{2}+1$};
\node[left] at (axis cs:9,30) {\tiny $x_{t+4}=t+2$};
\node[left] at (axis cs:9,42.5) {\tiny $x_{t+5}=t+3$};

\node[left] at (axis cs:9,82.5) {\tiny $x_{t+2}=\frac{t^2}{4}$};
\node[left] at (axis cs:9,95) {\tiny $x_{t+3}=\frac{t^2}{4}+\frac{t}{2}-1$};
\node[left] at (axis cs:9,107.5) {\tiny $x_{t+6}=\frac{t^2}{4}+\frac{5}{2}t+4$};
\node[left] at (axis cs:9,120) {\tiny $x_{t+7}=\frac{t^2}{4}+3t+5$};

\draw (axis cs:9,5) -- (axis cs:9.5,5) -- (axis cs:10.82432432,5.5000);
\draw (axis cs:9,17.5) -- (axis cs:9.5,17.5) -- (axis cs:10.82432432,6.5000);
\draw (axis cs:9,30) -- (axis cs:9.5,30) -- (axis cs:10.82432432,13);
\draw (axis cs:9,42.5) -- (axis cs:9.5,42.5) -- (axis cs:10.82432432,14);

\draw (axis cs:9,82.5) -- (axis cs:9.5,82.5) -- (axis cs:10.82432432,30.250);
\draw (axis cs:9,95) -- (axis cs:9.5,95) -- (axis cs:10.82432432,34.750);
\draw (axis cs:9,107.5) -- (axis cs:9.5,107.5) -- (axis cs:10.82432432,61.750);
\draw (axis cs:9,120) -- (axis cs:9.5,120) -- (axis cs:10.82432432,68.250);
\end{axis}
\end{tikzpicture}
\caption{\label{fig:regphase}\small A normal form orbit (purple) and the associated median sequence (blue) during the regular phase.}
\end{figure}
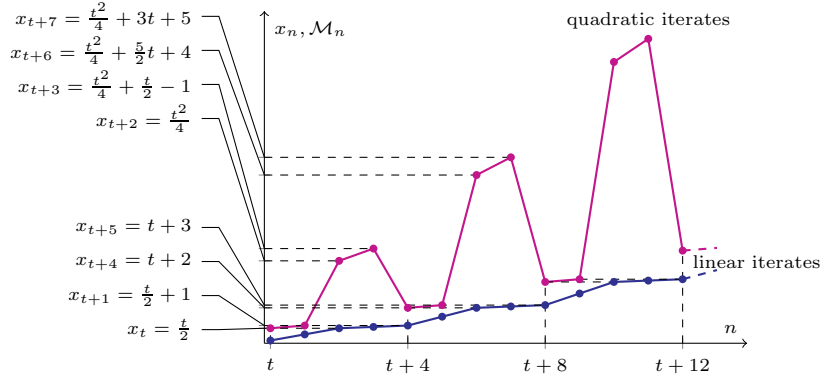
%%%%%%%%%%%%%%%%%%%%%%%%%%%%%%%%%%%%%%%%%%%%%%%%%%%%%%%%

During the regular phase, pairs of iterates which are linear and quadratic polynomials in $t$ are generated in alternation (figure \ref{fig:regphase}). 
These linear and quadratic iterates form two intertwined increasing sequences, 
each quadratic term lying above all linear ones:
$$\underbrace{0<1<x_t<x_{t+1}<x_{t+4}<x_{t+5}<\cdots}_{\text{linear iterates}}<\underbrace{x_{t+2}<x_{t+3}<x_{t+6}<x_{t+7}<\cdots}_{\text{quadratic iterates}}.$$
The non-decreasing median sequence 
$$\left(\mathcal{M}_{n}\right)_{n=t-2}^\infty=\left(0, \left\langle 0,1\right\rangle, 1, \left\langle 1,x_t\right\rangle, x_t, \left\langle x_t,x_{t+1}\right\rangle, x_{t+1}, \left\langle x_{t+1},x_{t+4}\right\rangle, x_{t+4}, \ldots\right),$$
walks firstly across the linear iterates until it eventually reaches the first quadratic iterate $x_{t+2}$, i.e., $\mathcal{M}_{N_t+2}=\left\langle x_{N_t-2}, x_{t+2}\right\rangle$, resulting in loss of regularity and concluding the regular phase. The explicit description is therefore known only up to time $N_t+2$. This number is a lower bound for the transit time of the orbit (figure \ref{fig:taunormal}) since stabilisation does not occur during the regular phase.

%%%%%%%%%%%%%%%%%%%%%%%%%%%%%%%%%%%%%%%%%%%%%%%%%%%%%%%% FIGURE
\begin{figure}
\centering
\input{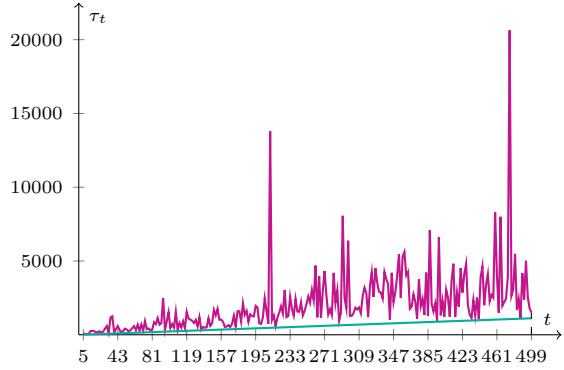}
\caption{\label{fig:taunormal}\small The transit time $\tau_t$ of the normal form orbit of order $t$ with its known lower bound $N_t+2$.}
\end{figure}
%%%%%%%%%%%%%%%%%%%%%%%%%%%%%%%%%%%%%%%%%%%%%%%%%%%%%%%%

A key observation ---which motivates our theory--- is that each pair of consecutive linear iterates $[0,1]$, $\left[x_t,x_{t+1}\right]$, $\left[x_{t+4},x_{t+5}\right]$, \ldots, $\left[x_{N_t+1},x_{N_t+2}\right]$ is generated when the median walks across the preceding pair\footnote{This means that the three consecutive medians involved in the computation of a pair are the smaller element, the mean, and the larger element of the preceding pair. For example, $\left[x_t,x_{t+1}\right]=\left[t\mathcal{M}_{t-1}-(t-1)\mathcal{M}_{t-2},(t+1)\mathcal{M}_{t}-t\mathcal{M}_{t-1}\right]$, where $\mathcal{M}_{t-2}=0$, $\mathcal{M}_{t-1}=\langle 0,1\rangle$, and $\mathcal{M}_t=1$.}, resulting in the following preservation of differences:
\begin{equation}\label{eq:gaptransfer}
1=x_{t+1}-x_t=x_{t+5}-x_{t+4}=\cdots=x_{N_t+2}-x_{N_t+1}.
\end{equation}
In this sense, the dynamics taking place in the regular phase can be regarded as \textit{recursive reproductions} of pairs of linear iterates.

The pairs of consecutive quadratic iterates, on the other hand, are generated when the median walks across the gaps between two consecutive pairs of linear iterates (figure \ref{fig:regphase}). For example, the quadratic pair $\left[x_{t+2},x_{t+3}\right]$ is generated when the median walks across the gap between $1$ and $x_t$. The fact that this gap is empty ---contains no other iterate--- implies that the median requires only two iterations to walk across it, thereby generating only two iterates which do not participate in the regular phase dynamics: $x_{t+2}$ and $x_{t+3}$. %non-participating 

In this paper we deal with phases which are not regular. 
Figure \ref{fig:205} suggests that there are two qualitatively different such phases: one which can be regarded as a perturbed version of the regular phase, and one which consists of straight segments of increasing length separated by conspicuous spikes. We refer to the former as an \textit{irregular phase} and the latter as a \textit{quasi-regular phase}.

In an irregular phase, the dynamics consists of the aforementioned recursive reproductions of pairs, but the gaps between consecutive pairs are not necessarily empty; they may contain previously-generated iterates (which we call \textit{obstacles}) which create irregularities in the orbit as the median walks across them. Later on we will construct a family of initial sets whose orbits contain only an irregular phase with an arbitrary number of obstacles in each gap. The transit time of such a family grows linearly as in part i) of our main theorem, where the coefficient depends on the number of obstacles.

In a quasi-regular phase, the dynamics is similar, but the reproducing subsets have larger cardinalities. Notice that, when a three-element subset, say, reproduces, the median walks across its elements requiring four iterations. Similarly, a four-element subset requires six iterations, a six-element subset requires ten iterations, and so on, producing the growth of the length of the straight segments in a quasi-regular phase visible in figure \ref{fig:205}. As a consequence, a family of initial sets whose orbits contain only a quasi-regular phase exhibits a quadratic growth of transit time as in part ii) of the theorem.\bigskip

This paper is organised as follows. In section \ref{section:ReproductionsOfSubsets} we introduce the notion of \textit{ready} subsets (abbreviated as \textit{R-subsets}) and state the \textit{reproduction lemma} which explains how an R-subset reproduces. Then we specialise this lemma to the cases in which the reproducing subset is an arithmetic progression and a pair (corollary \ref{cor:firstdifferenceaffine}).

In section \ref{section:Algorithms} we present two algorithms implementing recursive reproductions of these two types of subsets: algorithm A which generates a quasi-regular structure and algorithm B which generates a regular or an irregular structure depending on the presence or absence of obstacles. In section \ref{section:Proof}, we use these algorithms to give a constructive proof of our main theorem. 
In the last section, we show that part ii) of the main theorem gives a family 
of sets whose transit time grows with the complexity with exponent
which is larger than that observed on average in the system $[0,x,1]$ 
[figure \ref{fig:tauoriginal} (a)], but smaller than the exponent associated 
with the largest fluctuations [figure \ref{fig:tauoriginal} (b)].
\bigskip

\bigskip\noindent
{\sc Acknowledgements:} \/ 
The first author thanks Indonesia Endowment Fund for Education (LPDP) for the financial support.

%-------------------------------------------------------------------
\section{Reproduction of subsets}\label{section:ReproductionsOfSubsets}

Throughout this paper we assume that the median sequence is non-decreasing. A subset of $\left[x_1',\ldots,x_n'\right]$, where $n\geqslant 3$ is odd and $x_1'\leqslant\cdots\leqslant x_n'$, is \textit{ready} (or is an \textit{R-subset}) if
\begin{enumerate}
\item[i)] it is of the form $\left[x_i',\ldots,x_{i+k}'\right]$ for some $i\in\{1,\ldots,n\}$ and $k\in\{0,\ldots,n-i\}$;
\item[ii)] $x_1'=\mathcal{M}_n$;
\item[iii)] the sequence of \textit{first differences} $\left(\Delta x_j\right)_{j=i+1}^{i+k}$, where $\Delta x_j:=x_j-x_{j-1}$, is positive and non-decreasing.
\end{enumerate}
Notice that iii) implies that an R-subset contains no repetition. For example, $[4,6]$ and $[4,6,8]$ are R-subsets of $[2,2,3,4,6,8,9]$, whereas $[4,8]$, $[3,4,6]$, and $[4,6,8,9]$ are not, since they violate i) only, ii) only, and iii) only, respectively. Clearly, an R-subset has size at most $\frac{n+1}{2}$.

For convenience, let us agree that the elements of an R-subset are always indexed according to their ordering on the line. In other words, by writing an R-subset as $\left[u_0,\ldots,u_k\right]$ we mean that $u_0\leqslant\cdots\leqslant u_k$. In an \mmm\ orbit, an R-subset generates ---as the median walks across its elements--- a similar structure whose size is twice the number of its first differences. This is detailed in the following result.\bigskip

\begin{rl}
If $\zeta=\left[u_0,\ldots,u_k\right]\subseteq\xi_n$ is ready and $x_{n+1},x_{n+2}\geqslant u_k$, then for every $j\in\{2,\ldots,2k+1\}$ we have
\begin{equation}\label{eq:firstdifference}
x_{n+j}=\begin{cases}
\frac{\Delta u_\ell}{2}n+\ell \Delta u_\ell + u_{\ell-1}&\text{if }j=2\ell\text{ for some }\ell\in\mathbb{N}\\
\frac{\Delta u_\ell}{2}n+\ell \Delta u_\ell + u_\ell &\text{if }j=2\ell+1\text{ for some }\ell\in\mathbb{N}.
\end{cases}
\end{equation}
\end{rl}

\begin{proof}
Assume that $\zeta\subseteq\xi_n$ is ready and $x_{n+1},x_{n+2}\geqslant u_k$. We will prove by strong induction that
\begin{equation}\label{eq:prooffirstdifference}
\textnormal{For every }j\in\{2,\ldots,2k+1\}\textnormal{ we have \eqref{eq:firstdifference} and }x_{n+j}\geqslant u_k\textnormal{.}
\end{equation}

For the base case $j=2$, \eqref{eq:firstdifference} is true because
$$x_{n+2}=(n+2)\mathcal{M}_{n+1}-(n+1)\mathcal{M}_n=(n+2)\langle u_0,u_1\rangle-(n+1)u_0=\frac{\Delta u_1}{2}n+\ell \Delta u_1 + u_0,$$
where the condition $\mathcal{M}_{n+1}=\langle u_0,u_1\rangle$ follows from $x_{n+1}\geqslant u_k$ which is a part of our assumption. The inequality $x_{n+2}\geqslant u_k$ is also a part of our assumption.

Now suppose \eqref{eq:prooffirstdifference} holds for every $j\in\left\{2,\ldots,r-1\right\}$, for some $r\in\{3,\ldots,2k+1\}$. We shall show that it also holds for $j=r$. As a part of our assumption we have $x_{n+1}\geqslant u_k$. Combining this with the inductive hypothesis, we have $x_{n+1},\ldots,x_{n+r-1}\geqslant u_k$, so
$$\mathcal{M}_{n+r-1}=\begin{cases}
\langle u_{\ell-1},u_\ell\rangle&\text{if }r=2\ell\\
u_{\ell-1}&\text{if }r=2\ell+1
\end{cases}$$
and
$$\mathcal{M}_{n+r-2}=\begin{cases}
u_\ell&\text{if }r=2\ell\\
\langle u_{\ell-1},u_\ell\rangle&\text{if }r=2\ell+1.
\end{cases}$$
Using these, straightforward applications of \eqref{eq:MMMmedian} in both cases yield \eqref{eq:firstdifference}. Moreover, the fact that $\left(\Delta u_i\right)_{i=1}^k$ is positive and non-decreasing implies that $x_{n+r}>x_{n+r-1}$. Since the latter is at least $u_k$ by inductive hypothesis, the proof is complete.
\end{proof}\bigskip

\noindent Let us make some remarks regarding the reproduction lemma.

Property iii) in the definition of ready is necessary. Indeed, the above theorem is not valid for the subset $[0,100,101,110]$ of any $\xi_9$ for which $x_{10}\geqslant 110$. Notice that this subset satisfies i) and ii), but not iii) since its first differences form a non-monotonic sequence. As a consequence, we have $x_{11}=600\geqslant 110$, but $x_{12}=106<110$ which becomes an obstacle.

The sequence $\left(\Delta x_{n+j}\right)_{j=3}^{2k+1}$ of the first differences of the new iterates is \textit{positive but not necessarily non-decreasing}. Indeed, from \eqref{eq:firstdifference} we have that, for every $j\in\{3,\ldots,2k+1\}$,
\begin{equation}\label{eq:firstdifferenceresult}
\Delta x_{n+j}=\begin{cases}
\frac{\Delta u_\ell-\Delta u_{\ell-1}}{2}n+\ell\left(\Delta u_{\ell}-\Delta u_{\ell-1}\right)+\Delta u_{\ell-1}&\text{if }j=2\ell\\
\Delta u_\ell&\text{if }j=2\ell+1,
\end{cases}
\end{equation}
which is positive because the $\Delta u_\ell$s are positive and non-decreasing [property iii)]. Moreover, for the ready subset $[0,2,6]$ of any $\xi_9$ such that $x_{10}\geqslant 6$, we have $\left(x_{11},x_{12},x_{13},x_{14}\right)=(11,13,28,32)$ whose first differences form the non-monotonic sequence $(2,15,4)$.

If $\zeta$ is an arithmetic progression of order $d$ (a progression whose $n$-th term is given by a polynomial
of degree $d$ in $n$ \cite[theorem 6]{Dlab}), then the leading and constant coefficients of $x_{n+j}$ in \eqref{eq:firstdifference} are polynomials of degrees at most $d-1$ and at most $d$, respectively, whereas those of $\Delta x_{n+j}$ in \eqref{eq:firstdifferenceresult} are polynomials of degrees at most $d-2$ and $d-1$, respectively. In the rest of the paper we shall deal only with the case $d=1$.

In the case $d=1$, namely that in which $\zeta$ is an arithmetic progression (of order $1$), we shall prove that the condition $x_{n+2}\geqslant u_k$ always holds, and that the new points given by \eqref{eq:firstdifference} also form an arithmetic progression with the same modulus. This is stated in part i) of the following corollary, which will be the basis for the construction of a quasi-regular phase. Part ii) contains the special case $k=1$, which is the basis for the construction of a regular or an irregular phase, depending on the presence of obstacles.\bigskip

\begin{corollary}\label{cor:firstdifferenceaffine}
\textcolor{white}{a}
\begin{enumerate}
\item[i)] Suppose that $\zeta=\left[u_0,\ldots,u_k\right]\subseteq\xi_n$ is a ready arithmetic progression of modulus $b>0$. If $x_{n+1}\geqslant u_k$, then for every $j\in\{2,\ldots,2k+1\}$ we have
\begin{equation}\label{eq:corrAP}
x_{n+j}=b\left(\frac{n}{2}+j-1\right)+u_0.
\end{equation}
\item[ii)] Suppose that $\zeta=\left[u_0,u_1\right]\subseteq\xi_n$ is ready. If $x_{n+1}\geqslant u_1$, then we have
\begin{equation}\label{eq:twoelement}
x_{n+2}=\frac{u_1-u_0}{2}n+u_1\qquad\text{and}\qquad x_{n+3}=\frac{u_1-u_0}{2}n+2u_1-u_0.
\end{equation}
\end{enumerate}
\end{corollary}

\begin{proof}
First we prove part i). If the given conditions hold, then $\mathcal{M}_n=u_0$ and $\mathcal{M}_{n+1}=\langle u_0,u_0+b\rangle=u_0+\frac{b}{2}$ since $x_{n+1}\geqslant u_k$, so by \eqref{eq:MMMmedian} one obtains $x_{n+2}=b\left(\frac{n}{2}+1\right)+u_0\geqslant kb+u_0=u_k$ since $k\leqslant\frac{n-1}{2}$ and $b>0$. Hence, the reproduction lemma applies, giving part i). Setting $k=1$ gives part ii).
\end{proof}\bigskip

Notice that \eqref{eq:corrAP} can be rewritten as
$$x_{n+j}=\left(\frac{bn}{2}+u_0\right)+(j-1)b$$
showing that the location of the new arithmetic progression relative to the old one depends on $b$, $n$, and $u_0$. In addition, \eqref{eq:twoelement} implies
$$x_{n+3}-x_{n+2}=\left(\frac{u_1-u_0}{2}n+2u_1-u_0\right)-\left(\frac{u_1-u_0}{2}n+u_1\right)=u_1-u_0,$$
explaining \eqref{eq:gaptransfer} in a more general context.

%-------------------------------------------------------------------
\section{Recursive reproductions of subsets}\label{section:Algorithms}

We have seen in corollary \ref{cor:firstdifferenceaffine} that a ready arithmetic progression generates a new arithmetic progression as the median walks across its elements. Naturally, we are interested in the case where the new arithmetic progression becomes ready at a future time, and hence generates a third arithmetic progression in the orbit. The process can continue if the third progression eventually becomes ready and generates a fourth progression, and so on. In such circumstance we speak of \textit{recursive reproductions} of arithmetic progressions.

It turns out that recursive reproductions of arithmetic progressions 
occur in some normal form orbits, where they generate quasi-regular 
structures.
In section \ref{subs:ap} we distil this process in
an algorithm which describes it in a general setting (algorithm A). 
In section \ref{subs:pairs}, we then modify the algorithm to implement 
recursive reproductions of pairs (algorithm B).

Later we shall apply these algorithms to construct families of initial sets which establish our main theorem.

\subsection{Recursive reproductions of arithmetic progressions}\label{subs:ap} We begin with an example. Consider the normal form orbit of order $261$ (figure \ref{fig:261}). The set $\gamma_{641}$ contains the subset $\AP_0:=\left[x_{630},x_{637},x_{638}\right]$ which is a ready arithmetic progression. Since $x_{642}\geqslant x_{638}$, then the subset generates $$\AP_1:=\left[x_{643},x_{644},x_{645},x_{646}\right].$$

\begin{figure}%[h]
\centering
\input{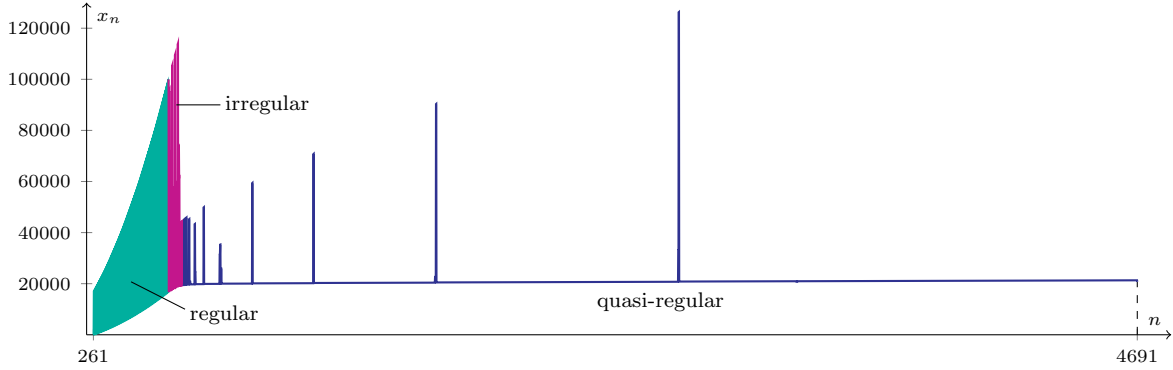}
\caption{\label{fig:261}\small The normal form orbit of order $261$ which features a regular, irregular, and quasi-regular structure.}
\end{figure}

Next, $\AP_1\subseteq\gamma_{647}$ is ready and $x_{647}\geqslant x_{646}$, so $\AP_1$ generates the next arithmetic progression $$\AP_2:=\left[x_{649},x_{650},x_{651},x_{652},x_{653},x_{654}\right].$$

\begin{figure}%[h]
\centering
\begin{tikzpicture}
\begin{axis}[
	xmin=628.3571429,
	xmax=731.2142857,
	ymin=15000,
	ymax=48000,
    ytick={20000,30000,40000},yticklabels={20000,30000,40000},
    xtick={629,641,647,728},
	axis lines=middle,
	samples=100,
	xlabel=$n$,
	ylabel=$x_n$,
	width=16cm,
	height=6cm,
	clip=false,
	axis lines=middle,
    x axis line style=->,
    y axis line style=->,
    scaled y ticks=false
]
\draw[dashed] (axis cs:629,15000)--(axis cs:629, 19401.12500);
\draw[dashed] (axis cs:641,15000)--(axis cs:641, 19641.50000);
\draw[dashed] (axis cs:647,15000)--(axis cs:647, 45241.93750);
\draw[dashed] (axis cs:728,15000)--(axis cs:728, 19824.75000);
\draw[thick,color=newpurple] plot coordinates {(axis cs:629, 19401.12500) (axis cs:630, 19401.87500) (axis cs:631, 43577.81250) (axis cs:632, 43655.18750) (axis cs:633, 19322.50000) (axis cs:634, 19322.75000) (axis cs:635, 44286.43750) (axis cs:636, 44365.31250) (axis cs:637, 19402.12500) (axis cs:638, 19402.37500) (axis cs:639, 44363.56250) (axis cs:640, 44441.93750) (axis cs:641, 19641.50000)};
\draw[thick,color=newblue] plot coordinates {(axis cs:641, 19641.50000) (axis cs:642, 19642.25000) (axis cs:643, 19482.25000)};
\draw[ultra thick,color=newlightblue] plot coordinates {(axis cs:643, 19482.25000) (axis cs:644, 19482.50000) (axis cs:645, 19482.75000) (axis cs:646, 19483.)}; %AP1 panjang 4
\node[below] at (axis cs:644.5, 19482.625) {\scriptsize $\AP_1$};
\draw[thick,color=newblue] plot coordinates {(axis cs:646, 19483.) (axis cs:647, 45241.93750) (axis cs:648, 45321.81250) (axis cs:649, 19563.37500)};
\draw[ultra thick,color=newlightblue] plot coordinates {(axis cs:649, 19563.37500) (axis cs:650, 19563.62500) (axis cs:651, 19563.87500) (axis cs:652, 19564.12500) (axis cs:653, 19564.37500) (axis cs:654, 19564.62500)}; %AP2 panjang 6
\node[below] at (axis cs:651.5,19564) {\scriptsize $\AP_2$};
\draw[thick,color=newblue] plot coordinates {(axis cs:654, 19564.62500) (axis cs:655, 45805.81250) (axis cs:656, 45886.18750) (axis cs:657, 19645.50000)};
\draw[ultra thick,color=newlightblue] plot coordinates {(axis cs:657, 19645.50000) (axis cs:658, 19645.75000) (axis cs:659, 19646.) (axis cs:660, 19646.25000) (axis cs:661, 19646.50000) (axis cs:662, 19646.75000) (axis cs:663, 19647.) (axis cs:664, 19647.25000) (axis cs:665, 19647.50000) (axis cs:666, 19647.75000)}; %AP3 panjang 10
\node[below] at (axis cs:661.5, 19646.625) {\scriptsize $\AP_3$};
\draw[thick,color=newblue] plot coordinates {(axis cs:666, 19647.75000) (axis cs:667, 45202.43750) (axis cs:668, 45279.31250) (axis cs:669, 19892.37500) (axis cs:670, 19893.12500) (axis cs:671, 20732.62500) (axis cs:672, 20735.87500) (axis cs:673, 19729.62500)};
\draw[ultra thick,color=newlightblue] plot coordinates {(axis cs:673, 19729.62500) (axis cs:674, 19729.87500) (axis cs:675, 19730.12500) (axis cs:676, 19730.37500) (axis cs:677, 19730.62500) (axis cs:678, 19730.87500) (axis cs:679, 19731.12500) (axis cs:680, 19731.37500) (axis cs:681, 19731.62500) (axis cs:682, 19731.87500) (axis cs:683, 19732.12500) (axis cs:684, 19732.37500) (axis cs:685, 19732.62500) (axis cs:686, 19732.87500) (axis cs:687, 19733.12500) (axis cs:688, 19733.37500) (axis cs:689, 19733.62500) (axis cs:690, 19733.87500)}; %AP4 panjang 18
\node[below] at (axis cs:681.5,19731.75) {\scriptsize $\AP_4$};
\draw[thick,color=newblue] plot coordinates {(axis cs:690, 19733.87500) (axis cs:691, 43314.50000) (axis cs:692, 43383.) (axis cs:693, 24350.68750) (axis cs:694, 24364.06250) (axis cs:695, 19816.50000)};
\draw[ultra thick,color=newlightblue] plot coordinates {(axis cs:695, 19816.50000) (axis cs:696, 19816.75000) (axis cs:697, 19817.) (axis cs:698, 19817.25000) (axis cs:699, 19817.50000) (axis cs:700, 19817.75000) (axis cs:701, 19818.) (axis cs:702, 19818.25000) (axis cs:703, 19818.50000) (axis cs:704, 19818.75000) (axis cs:705, 19819.) (axis cs:706, 19819.25000) (axis cs:707, 19819.50000) (axis cs:708, 19819.75000) (axis cs:709, 19820.) (axis cs:710, 19820.25000) (axis cs:711, 19820.50000) (axis cs:712, 19820.75000) (axis cs:713, 19821.) (axis cs:714, 19821.25000) (axis cs:715, 19821.50000) (axis cs:716, 19821.75000) (axis cs:717, 19822.) (axis cs:718, 19822.25000) (axis cs:719, 19822.50000) (axis cs:720, 19822.75000) (axis cs:721, 19823.) (axis cs:722, 19823.25000) (axis cs:723, 19823.50000) (axis cs:724, 19823.75000) (axis cs:725, 19824.) (axis cs:726, 19824.25000) (axis cs:727, 19824.50000) (axis cs:728, 19824.75000)}; %AP5 panjang 34
\node[below] at (axis cs:711.5,19820.625) {\scriptsize $\AP_5$};

\fill[newlightblue] (axis cs:630, 19401.87500) circle (2pt);
\fill[newlightblue] (axis cs:637, 19402.12500) circle (2pt);
\fill[newlightblue] (axis cs:638, 19402.37500) circle (2pt);
\node[below] at (axis cs:634,19402) {\scriptsize $\AP_0$};
\end{axis}
\end{tikzpicture}
\caption{\label{fig:QR261}\small
First few successive arithmetic progressions in the quasi-regular phase of the normal form orbit of order $261$.}
\end{figure}
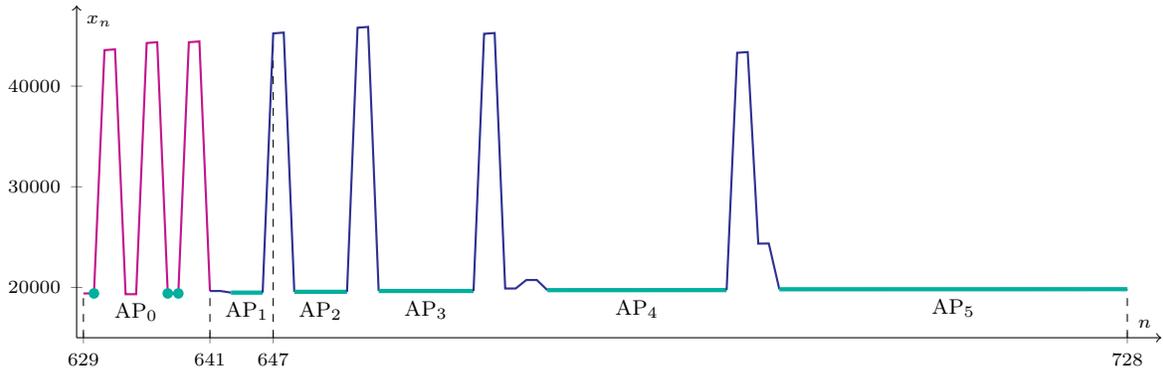

For this particular orbit, the process continues in the same way until ten arithmetic progressions $\AP_1$, \ldots, $\AP_{10}$ are generated\footnote{$\AP_{10}=\left[x_{1717},\ldots,x_{2742}\right]\subseteq\gamma_{2740}$ is not intact since $x_{1966}<x_{284}<x_{1967}$, hence it never becomes ready.} (the first five are shown in figure \ref{fig:QR261}), with the sequence of time steps $n_i$ at which $\AP_i$ is ready being
$$\left(n_0,n_1,\ldots,n_{9}\right)=\left(641, 647, 655, 671, 693, 729, 803, 935, 1195, 1715\right).$$
From the corollary, these arithmetic progressions have the same modulus ---which is $\frac{1}{4}$--- and their lengths grow exponentially.
$$
\begin{array}{l|cccccccccc}
\text{arithmetic progression} & \AP_1 & \AP_2& \AP_3& \AP_4& \AP_5& \AP_6& \AP_7& \AP_8& \AP_9& \AP_{10}\\
\hline
\text{first term} &   x_{643}& x_{649}& x_{657}& x_{673}& x_{695}& x_{731}& x_{805}& x_{937}& x_{1197}& x_{1717}\\
\hline
\text{last term} &   x_{646}& x_{654}& x_{666}& x_{690}& x_{728}& x_{796}& x_{934}& x_{1194}& x_{1710}& x_{2742}\\
\hline
\text{length} &   4& 6& 10& 18& 34& 66& 130& 258& 514& 1026
\end{array}
$$

Motivated by the above example, let us now develop an algorithm 
which implements recursive reproductions of arithmetic progressions 
in a general orbit, and produces explicit formulae for all the 
straight segments in the generated quasi-regular structure. 
The input is a set containing a ready arithmetic progression of length three. 
Since the \mmm\ preserves affine-equivalence\footnote{If $\left[x_1,\ldots,x_n\right]\mapsto\left[x_1,\ldots,x_n,x_{n+1}\right]$ then $\left[ax_1+b,\ldots,ax_n+b\right]\mapsto\left[ax_1+b,\ldots,ax_n+b,ax_{n+1}+b\right]$ for every $a,b\in\mathbb{R}$.}, 
we can assume that the arithmetic progression is $[0,1,2]$. 
The algorithm consists of a loop. At each iteration, it generates
the arithmetic progression prescribed by part i) of corollary 
\ref{cor:firstdifferenceaffine} and checks whether this progression 
becomes ready at some future time step. If so, the next iteration is performed, otherwise the algorithm stops.\bigskip

\begin{algorithma}[Recursive reproductions of arithmetic progressions]\textcolor{white}{a}
\begin{itemize}[leftmargin=2cm,itemsep=0pt]
\item[{\sc input:}] A set $\xi_{n_0}$ with R-subset $\AP_0=[0,1,2]$, where $n_{0}\geqslant 5$ is odd and $x_{n_0+1}\geqslant\max\left(\AP_0\right)$.
\item[{\sc output:}] Sequences $\left(n_0,\ldots,n_{N-1}\right)$ and
$\left(\AP_1,\ldots,\AP_N\right)=\left(\left[x_{n_i+2},\ldots,x_{n_i+2\left|\AP_i\right|-1}\right]\right)_{i=0}^{N-1}$,
where
\begin{enumerate}
\item[i)] For every $i\in\{0,\ldots,N-1\}$, $\AP_i\subseteq \xi_{n_i}$ is ready and generates
$\AP_{i+1}$ according to part i) of corollary \ref{cor:firstdifferenceaffine}.
\item[ii)] $N$ is the largest positive integer $i$ such that $\AP_i$ exists.
\end{enumerate}
\end{itemize}
\begin{enumerate}[font=\scriptsize]
\item[\textnormal{1}\,\,\,\,] $\textnormal{go-on}:=\textit{true}$
\item[\textnormal{2}\,\,\,\,] $i:=0$
\item[\textnormal{3}\,\,\,\,] \textnormal{\textbf{while $\textnormal{go-on}$ do}}
\item[\textnormal{4}\,\,\,\,] \hspace*{0.4cm} \textnormal{\textbf{for $j\in\left\{2,\ldots,2\left|\AP_i\right|-1\right\}$ do}}
\item[\textnormal{5}\,\,\,\,] \hspace*{0.4cm}\hspace*{0.4cm} $x_{n_i+j}:=j+i-1+\frac{1}{2}\sum_{\ell=0}^i n_\ell$
\item[\textnormal{6}\,\,\,\,] \hspace*{0.4cm} \textnormal{\textbf{end do}}
\item[\textnormal{7}\,\,\,\,] \hspace*{0.4cm} $\AP_{i+1}:=\left[x_{n_i+2},\ldots,x_{n_i+2\left|\AP_i\right|-1}\right]$
\item[\textnormal{8}\,\,\,\,] \hspace*{0.4cm} $k:=n_{i}+2$
\item[\textnormal{9}\,\,\,\,] \hspace*{0.4cm} $\textnormal{go-on}:=\textit{false}$
\item[\textnormal{10}\,\,\,\,] \hspace*{0.4cm} \textnormal{\textbf{while $\mathcal{M}_{k-1}\neq\mathcal{M}_k$ and $\mathcal{M}_k\leqslant x_{n_i+2}$ do}}
\item[\textnormal{11}\,\,\,\,] \hspace*{0.4cm}\hspace*{0.4cm} \textnormal{\textbf{if $\AP_{i+1}\subseteq\xi_k$\textnormal{ is ready} and $x_{k+1}\geqslant\max\left(\AP_{i+1}\right)$ then}}
\item[\textnormal{12}\,\,\,\,] \hspace*{0.4cm}\hspace*{0.4cm} \hspace*{0.4cm} $n_{i+1}:=k$
\item[\textnormal{13}\,\,\,\,] \hspace*{0.4cm}\hspace*{0.4cm} \hspace*{0.4cm} $\textnormal{go-on}:=\textit{true}$
\item[\textnormal{14}\,\,\,\,] \hspace*{0.4cm}\hspace*{0.4cm} \textnormal{\textbf{end if}}
\item[\textnormal{15}\,\,\,\,] \hspace*{0.4cm}\hspace*{0.4cm} $k:=k+2$
\item[\textnormal{16}\,\,\,\,] \hspace*{0.4cm} \textnormal{\textbf{end do}}
\item[\textnormal{17}\,\,\,\,] \hspace*{0.4cm} $i:=i+1$
\item[\textnormal{18}\,\,\,\,] \textnormal{\textbf{end do}}
\item[\textnormal{19}\,\,\,\,] \textnormal{\textbf{return} $\left(\left(n_0,\ldots,n_{N-1}\right),\left(\AP_1,\ldots,\AP_N\right)\right)$}
\end{enumerate}
\end{algorithma}\bigskip

Notice that an input set with the prescribed properties exists, e.g., $[-3,-3,0,1,2]$.
We now prove the correctness of the algorithm.\bigskip

\begin{proposition}
Algorithm A is correct and terminates if the input set stabilises.
\end{proposition}

\begin{proof}
Let $\mathcal{L}$ be the boolean expression
$$
[i=0]\lor\left[\left(\AP_{i-1}\subseteq\xi_{n_{i-1}}\text{ is ready}\right)
 \land \left(\AP_{i-1}\text{ generates }\AP_i\right)\right].
$$
We shall prove that $\mathcal{L}$ is a loop invariant for the outer 
while-loop (lines 3 to 18) of the algorithm.
Consider such a loop.
At the start of the first execution we have $i=0$, so $\mathcal{L}$ holds.
We next prove that $\mathcal{L}$ holds at the end of every execution 
of the statement-sequence of the loop. 
We split the proof into two cases:\bigskip

\noindent\textsc{\underline{Case I}}: At the start of the loop 
we have $i=0$. In this case $\text{go-on}=\textit{true}$, and the algorithm computes
$$\AP_1=\left[x_{n_0+2},\ldots,x_{n_0+5}\right],$$
where $x_{n_0+j}=\frac{n_0}{2}+j-1$ for every $j\in\{2,3,4,5\}$. 
At the end of the execution we have $i=1$, so the left expression 
in $\mathcal{L}$ is false, but the right one is true because 
$\AP_0\subseteq\xi_{n_0}$ is ready by the 
prescribed properties of $\xi_{n_0}$, and $\AP_0$
generates $\AP_1$, as easily checked. At the end of the loop, if $\text{go-on}=\textit{false}$, then the algorithm terminates. Otherwise, the algorithm continues and we proceed to Case II.\medskip

\noindent\textsc{\underline{Case II}}: At the start of the loop 
we have $i=i'\geqslant 1$.
At the end of the loop we have $i=i'+1\geqslant2$; since the left 
expression in $\mathcal{L}$ is false, we must show that the conjunction
on the right is true, namely that $\AP_{i'}\subseteq\xi_{n_{i'}}$ is ready 
and generates $\AP_{i'+1}$ according to part i) of corollary 
\ref{cor:firstdifferenceaffine}.
Since $\text{go-on}=\textit{true}$, in the previous execution the algorithm must have entered the
inner while-loop (lines 10 to 16) and executed the statement-sequence of the if-structure,
which means that $\AP_{i'}\subseteq\xi_{n_{i'}}$ is ready, proving the first half of the desired statement.
It now remains to prove that $\AP_{i'+1}$ agrees with part i) of corollary \ref{cor:firstdifferenceaffine}.
From the previous execution we know that
$$\AP_{i'}=\left[x_{n_{i'-1}+2},\ldots,x_{n_{i'-1}+2\left|\AP_{i'-1}\right|-1}\right],$$
where $x_{n_{i'-1}+j}=j+i'-2+\frac{1}{2}\sum_{\ell=0}^{i'-1} n_\ell$ for every
$j\in\left\{2,\ldots,2\left|\AP_{i'-1}\right|-1\right\}$, is an arithmetic progression with
modulus $1$ and first term $i'+\frac{1}{2}\sum_{\ell=0}^{i'-1} n_\ell$.
Therefore, by part i) of corollary \ref{cor:firstdifferenceaffine} (which can be applied because the first
half of the desired statement holds), we obtain
$$x_{n_{i'}+j}=1\cdot\left(\frac{n_{i'}}{2}+j-1\right)+\left(i'+\frac{1}{2}\sum_{\ell=0}^{i'-1} n_\ell\right)
=j+i'-1+\frac{1}{2}\sum_{\ell=0}^{i'} n_\ell$$
for every $j\in\{2,\ldots,2\left|\AP_{i'}\right|-1\}$, which is precisely how the algorithm
computes $\AP_{i'+1}$ in the execution being considered.
The proof that $\mathcal{L}$ is a loop invariant is complete.\bigskip

Next we prove that if the orbit of the input set stabilises, then
the algorithm terminates. 
We begin to show that each time the statement-sequence of the outer 
while-loop is executed, the value of $k$ assigned on line 8 increases (if previously assigned). 
The first execution always takes place. For any subsequent execution
to take place, the loop-control variable go-on must have value 
\textit{true}, which requires the flow to enter the inner while-loop 
and execute the statement-sequence of the if-structure.
Consequently, at the end of this previous iteration, the value of the 
recently-created variable $n_i$ is at least the value of $k$ as assigned 
on line 8 (in this previous execution).
In the current execution, line 8 then assigns to $k$ the value $n_i+2$, which is therefore larger than
its value defined on the same line in the previous execution. Since the input set stabilises,
then the algorithm must eventually reach a value of $k$ for which $\mathcal{M}_{k-1}=\mathcal{M}_k$.
In an execution where such a value of $k$ is reached, the inner while-loop is not performed,
and hence the execution ends with $\text{go-on}=\textit{false}$, terminating the algorithm.

Since the algorithm terminates, the integer $N$ stated in the {\sc output} exists.
Since the second alternative in $\mathcal{L}$ is proved to hold at the end of every iteration
which is not the first one, we have also proved that the algorithm returns the correct output.
\end{proof}\bigskip

In the algorithm we have $\left|\AP_0\right|=3$, and, by part i) of corollary \ref{cor:firstdifferenceaffine},
$\left|\AP_i\right|=2\left|\AP_{i-1}\right|-2$ for every $i\in\{1,\ldots,N\}$.
Solving this recursion gives
$$\left|\AP_i\right|=2^i+2\qquad\text{and}\qquad\left|\AP_0\right|+\left|\AP_1\right|+\cdots+\left|\AP_i\right|=2^{i+1}+2i-2$$
for every $i\in\{0,\ldots,N\}$.

\subsection{Recursive reproductions of pairs}\label{subs:pairs}

Let us now turn to recursive reproductions of pairs. We already know that such reproductions take place in the regular phase of a normal form orbit. Here, the recursive reproductions of pairs of linear iterates 
$$\left[x_t,x_{t+1}\right]\mapsto\left[x_{t+4},x_{t+5}\right]\mapsto \ldots\mapsto\left[x_{N_t+1},x_{N_t+2}\right],$$
where $t$ is the order, is based on the fact that, for every $\ell\in\{0,\ldots,\frac{N_t-t-3}{4}\}$, the pair $\left[x_{t+4\ell},x_{t+4\ell+1}\right]\subseteq \gamma_{t+4\ell+2}$ is ready and, since $x_{t+4\ell+3}\geqslant x_{t+4\ell+1}$, generates the pair $\left[x_{t+4\ell+4},x_{t+4\ell+5}\right]$.

In a regular phase, however, there are no obstacles. Our aim now is to modify algorithm A to implement recursive reproductions of pairs which, in general, may be perturbed by obstacles. To denote the reproducing pairs we will use $\PP_0$, \ldots, $\PP_N$ which will replace $\AP_0$, \ldots, $\AP_N$ in the algorithm.

Two modifications in the algorithm are required. First, the R-subset $\AP_0=[0,1,2]$ of the input set $\xi_{n_0}$ needs to be replaced by $\PP_0=[0,1]$. As a consequence, the for-loop (lines 4 to 6) consists only of two iterations, so it can be rewritten more simply as $x_{n_i+2}:=i+1+\frac{1}{2}\sum_{\ell=0}^i n_\ell$ and $x_{n_i+3}:=x_{n_i+2}+1$.

The second modification concerns termination of the algorithm. In the case of a quasi-regular phase, the algorithm terminates if stabilisation has been reached or the latest arithmetic progression never becomes ready (line 10). An irregular phase, however, does not typically end with stabilisation (recall figures \ref{fig:205} and \ref{fig:261}). Therefore, the condition $\mathcal{M}_{k-1}\neq\mathcal{M}_k$ (line 10) needs a replacement. In order to find a suitable replacement condition, let us study the first irregular phase in the normal form orbit of order $205$ (figure \ref{fig:205}).

For this orbit we have $N_{205}=456$, and hence the last reproduction of linear pairs in the regular phase is $\left[x_{453},x_{454}\right]\mapsto\left[x_{457},x_{458}\right]$. Now the elements of $\gamma_{458}$ are
$$0<1<\cdots<\underbrace{x_{449}<x_{450}}_{\text{a linear pair}}<\underbrace{x_{453}<x_{454}}_{\text{a linear pair}}<x_{207}<\underbrace{x_{457}<x_{458}}_{\text{a linear pair}}<x_{208}<\cdots<x_{455}<x_{456}.$$
Notice that the gap between the linear pairs $\left[x_{453},x_{454}\right]$ and $\left[x_{457},x_{458}\right]$ is not empty; it contains an obstacle, namely the first quadratic iterate $x_{207}$. Consequently, the median sequence requires more than two iterations to walk from the former pair to the latter. In fact, we have
$$\left(\mathcal{M}_n\right)_{n=457}^{461}=\left(x_{454},\left\langle x_{454},x_{207}\right\rangle,x_{207},\left\langle x_{207},x_{457}\right\rangle,x_{457}\right).$$
These medians generate four iterates
$$
\begin{array}{rcl}
x_{459}&=&-458\mathcal{M}_{457}+459\mathcal{M}_{458},\\
x_{460}&=&-459\mathcal{M}_{458}+460\mathcal{M}_{459},
\end{array}\qquad
\begin{array}{rcl}
x_{461}&=&-460\mathcal{M}_{459}+461\mathcal{M}_{460},\\
x_{462}&=&-461\mathcal{M}_{460}+462\mathcal{M}_{461}
\end{array}
$$
before the linear pair $\left[x_{457},x_{458}\right]$ becomes an R-subset of $\gamma_{461}$ and generates the next linear pair $\left[x_{463},x_{464}\right]$. See figure \ref{fig:205early}.

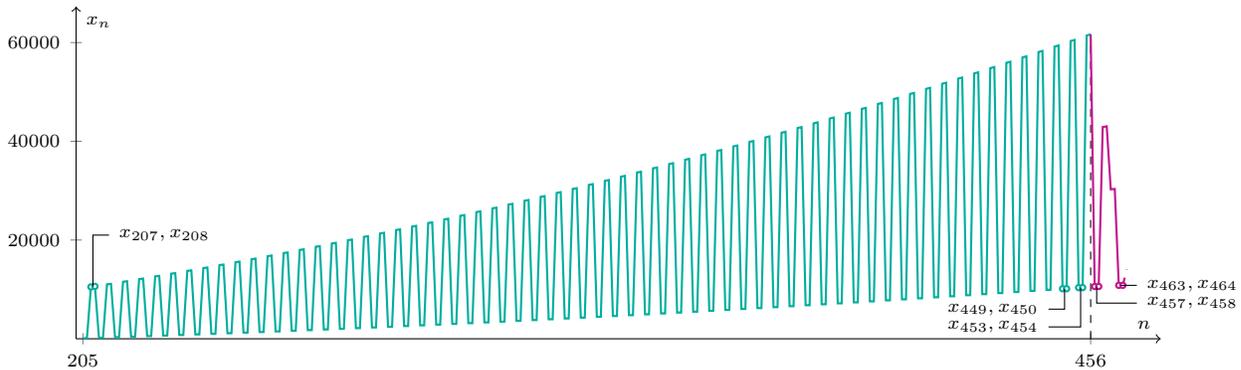
\begin{figure}
\centering
\begin{tikzpicture}
\begin{axis}[
	xmin=203.3116883,
	xmax=473.4415584,
	ymin=0,
	ymax=67270.63636,
    ytick={20000,40000,60000},yticklabels={20000,40000,60000},
    xtick={205,456},xticklabels={205,456},
	axis lines=middle,
	samples=100,
	xlabel=$n$,
	ylabel=$x_n$,
	width=16cm,
	height=6cm,
	clip=false,
	axis lines=middle,
    x axis line style=->,
    y axis line style=->,
    scaled y ticks=false
]
\draw[dashed] (axis cs:456,0)--(axis cs:456., 61664.750);

\draw[thick,color=newlightblue] plot coordinates {(axis cs:205., 102.50000) (axis cs:206., 103.50000) (axis cs:207., 10506.250) (axis cs:208., 10607.750) (axis cs:209., 207.) (axis cs:210., 208.) (axis cs:211., 11022.750) (axis cs:212., 11126.250) (axis cs:213., 313.50000) (axis cs:214., 314.50000) (axis cs:215., 11549.250) (axis cs:216., 11654.750) (axis cs:217., 422.) (axis cs:218., 423.) (axis cs:219., 12085.750) (axis cs:220., 12193.250) (axis cs:221., 532.50000) (axis cs:222., 533.50000) (axis cs:223., 12632.250) (axis cs:224., 12741.750) (axis cs:225., 645.) (axis cs:226., 646.) (axis cs:227., 13188.750) (axis cs:228., 13300.250) (axis cs:229., 759.50000) (axis cs:230., 760.50000) (axis cs:231., 13755.250) (axis cs:232., 13868.750) (axis cs:233., 876.) (axis cs:234., 877.) (axis cs:235., 14331.750) (axis cs:236., 14447.250) (axis cs:237., 994.50000) (axis cs:238., 995.50000) (axis cs:239., 14918.250) (axis cs:240., 15035.750) (axis cs:241., 1115.) (axis cs:242., 1116.) (axis cs:243., 15514.750) (axis cs:244., 15634.250) (axis cs:245., 1237.5000) (axis cs:246., 1238.5000) (axis cs:247., 16121.250) (axis cs:248., 16242.750) (axis cs:249., 1362.) (axis cs:250., 1363.) (axis cs:251., 16737.750) (axis cs:252., 16861.250) (axis cs:253., 1488.5000) (axis cs:254., 1489.5000) (axis cs:255., 17364.250) (axis cs:256., 17489.750) (axis cs:257., 1617.) (axis cs:258., 1618.) (axis cs:259., 18000.750) (axis cs:260., 18128.250) (axis cs:261., 1747.5000) (axis cs:262., 1748.5000) (axis cs:263., 18647.250) (axis cs:264., 18776.750) (axis cs:265., 1880.) (axis cs:266., 1881.) (axis cs:267., 19303.750) (axis cs:268., 19435.250) (axis cs:269., 2014.5000) (axis cs:270., 2015.5000) (axis cs:271., 19970.250) (axis cs:272., 20103.750) (axis cs:273., 2151.) (axis cs:274., 2152.) (axis cs:275., 20646.750) (axis cs:276., 20782.250) (axis cs:277., 2289.5000) (axis cs:278., 2290.5000) (axis cs:279., 21333.250) (axis cs:280., 21470.750) (axis cs:281., 2430.) (axis cs:282., 2431.) (axis cs:283., 22029.750) (axis cs:284., 22169.250) (axis cs:285., 2572.5000) (axis cs:286., 2573.5000) (axis cs:287., 22736.250) (axis cs:288., 22877.750) (axis cs:289., 2717.) (axis cs:290., 2718.) (axis cs:291., 23452.750) (axis cs:292., 23596.250) (axis cs:293., 2863.5000) (axis cs:294., 2864.5000) (axis cs:295., 24179.250) (axis cs:296., 24324.750) (axis cs:297., 3012.) (axis cs:298., 3013.) (axis cs:299., 24915.750) (axis cs:300., 25063.250) (axis cs:301., 3162.5000) (axis cs:302., 3163.5000) (axis cs:303., 25662.250) (axis cs:304., 25811.750) (axis cs:305., 3315.) (axis cs:306., 3316.) (axis cs:307., 26418.750) (axis cs:308., 26570.250) (axis cs:309., 3469.5000) (axis cs:310., 3470.5000) (axis cs:311., 27185.250) (axis cs:312., 27338.750) (axis cs:313., 3626.) (axis cs:314., 3627.) (axis cs:315., 27961.750) (axis cs:316., 28117.250) (axis cs:317., 3784.5000) (axis cs:318., 3785.5000) (axis cs:319., 28748.250) (axis cs:320., 28905.750) (axis cs:321., 3945.) (axis cs:322., 3946.) (axis cs:323., 29544.750) (axis cs:324., 29704.250) (axis cs:325., 4107.5000) (axis cs:326., 4108.5000) (axis cs:327., 30351.250) (axis cs:328., 30512.750) (axis cs:329., 4272.) (axis cs:330., 4273.) (axis cs:331., 31167.750) (axis cs:332., 31331.250) (axis cs:333., 4438.5000) (axis cs:334., 4439.5000) (axis cs:335., 31994.250) (axis cs:336., 32159.750) (axis cs:337., 4607.) (axis cs:338., 4608.) (axis cs:339., 32830.750) (axis cs:340., 32998.250) (axis cs:341., 4777.5000) (axis cs:342., 4778.5000) (axis cs:343., 33677.250) (axis cs:344., 33846.750) (axis cs:345., 4950.) (axis cs:346., 4951.) (axis cs:347., 34533.750) (axis cs:348., 34705.250) (axis cs:349., 5124.5000) (axis cs:350., 5125.5000) (axis cs:351., 35400.250) (axis cs:352., 35573.750) (axis cs:353., 5301.) (axis cs:354., 5302.) (axis cs:355., 36276.750) (axis cs:356., 36452.250) (axis cs:357., 5479.5000) (axis cs:358., 5480.5000) (axis cs:359., 37163.250) (axis cs:360., 37340.750) (axis cs:361., 5660.) (axis cs:362., 5661.) (axis cs:363., 38059.750) (axis cs:364., 38239.250) (axis cs:365., 5842.5000) (axis cs:366., 5843.5000) (axis cs:367., 38966.250) (axis cs:368., 39147.750) (axis cs:369., 6027.) (axis cs:370., 6028.) (axis cs:371., 39882.750) (axis cs:372., 40066.250) (axis cs:373., 6213.5000) (axis cs:374., 6214.5000) (axis cs:375., 40809.250) (axis cs:376., 40994.750) (axis cs:377., 6402.) (axis cs:378., 6403.) (axis cs:379., 41745.750) (axis cs:380., 41933.250) (axis cs:381., 6592.5000) (axis cs:382., 6593.5000) (axis cs:383., 42692.250) (axis cs:384., 42881.750) (axis cs:385., 6785.) (axis cs:386., 6786.) (axis cs:387., 43648.750) (axis cs:388., 43840.250) (axis cs:389., 6979.5000) (axis cs:390., 6980.5000) (axis cs:391., 44615.250) (axis cs:392., 44808.750) (axis cs:393., 7176.) (axis cs:394., 7177.) (axis cs:395., 45591.750) (axis cs:396., 45787.250) (axis cs:397., 7374.5000) (axis cs:398., 7375.5000) (axis cs:399., 46578.250) (axis cs:400., 46775.750) (axis cs:401., 7575.) (axis cs:402., 7576.) (axis cs:403., 47574.750) (axis cs:404., 47774.250) (axis cs:405., 7777.5000) (axis cs:406., 7778.5000) (axis cs:407., 48581.250) (axis cs:408., 48782.750) (axis cs:409., 7982.) (axis cs:410., 7983.) (axis cs:411., 49597.750) (axis cs:412., 49801.250) (axis cs:413., 8188.5000) (axis cs:414., 8189.5000) (axis cs:415., 50624.250) (axis cs:416., 50829.750) (axis cs:417., 8397.) (axis cs:418., 8398.) (axis cs:419., 51660.750) (axis cs:420., 51868.250) (axis cs:421., 8607.5000) (axis cs:422., 8608.5000) (axis cs:423., 52707.250) (axis cs:424., 52916.750) (axis cs:425., 8820.) (axis cs:426., 8821.) (axis cs:427., 53763.750) (axis cs:428., 53975.250) (axis cs:429., 9034.5000) (axis cs:430., 9035.5000) (axis cs:431., 54830.250) (axis cs:432., 55043.750) (axis cs:433., 9251.) (axis cs:434., 9252.) (axis cs:435., 55906.750) (axis cs:436., 56122.250) (axis cs:437., 9469.5000) (axis cs:438., 9470.5000) (axis cs:439., 56993.250) (axis cs:440., 57210.750) (axis cs:441., 9690.) (axis cs:442., 9691.) (axis cs:443., 58089.750) (axis cs:444., 58309.250) (axis cs:445., 9912.5000) (axis cs:446., 9913.5000) (axis cs:447., 59196.250) (axis cs:448., 59417.750) (axis cs:449., 10137.) (axis cs:450., 10138.) (axis cs:451., 60312.750) (axis cs:452., 60536.250) (axis cs:453., 10363.500) (axis cs:454., 10364.500) (axis cs:455., 61439.250) (axis cs:456., 61664.750)};

\draw[thick,color=newpurple] plot coordinates {(axis cs:456., 61664.750) (axis cs:457., 10592.) (axis cs:458., 10593.) (axis cs:459., 42896.12500) (axis cs:460., 43037.87500) (axis cs:461., 30271.62500) (axis cs:462., 30357.37500) (axis cs:463., 10823.50000) (axis cs:464., 10824.50000)};

\draw[thick,dashed,newpurple] (axis cs:464., 10824.50000) -- (axis cs:465., 14022.37500);

\draw[thick,newlightblue,fill=white] (axis cs:453., 10363.500) circle (1pt);
\draw[thick,newlightblue,fill=white] (axis cs:454., 10364.500) circle (1pt);

\draw[thick,newlightblue,fill=white] (axis cs:449., 10137.) circle (1pt);
\draw[thick,newlightblue,fill=white] (axis cs:450., 10138.) circle (1pt);

\draw[thick,newlightblue,fill=white] (axis cs:207., 10506.250) circle (1pt);
\draw[thick,newlightblue,fill=white] (axis cs:208., 10607.750) circle (1pt);

\draw[thick,newpurple,fill=white] (axis cs:457., 10592.) circle (1pt);
\draw[thick,newpurple,fill=white] (axis cs:458., 10593.) circle (1pt);

\draw[thick,newpurple,fill=white] (axis cs:463., 10823.50000) circle (1pt);
\draw[thick,newpurple,fill=white] (axis cs:464., 10824.50000) circle (1pt);

\draw (axis cs:207.5,10557.00000) -- (axis cs:207.5,21000) -- (axis cs:211.5,21000) node[right] {\tiny $x_{207},x_{208}$};
\draw (axis cs:453.5,10364) -- (axis cs:453.5,2400) -- (axis cs:445.5,2400) node[left] {\tiny $x_{453},x_{454}$};
\draw (axis cs:449.5,10137.5) -- (axis cs:449.5,6000) -- (axis cs:445.5,6000) node[left] {\tiny $x_{449},x_{450}$};
\draw (axis cs:457.5,10592.5) -- (axis cs:457.5,7224) -- (axis cs:467.5,7224) node[right] {\tiny $x_{457},x_{458}$};
\draw (axis cs:463.5,10824) -- (axis cs:467.5,10824) node[right] {\tiny $x_{463},x_{464}$};
\end{axis}
\end{tikzpicture}
\caption{\label{fig:205early}\small Early evolution in normal form orbit of order $205$.}
\end{figure}

In the above dynamics, the gap between the linear pairs $\left[x_{449},x_{450}\right]$ and $\left[x_{457},x_{458}\right]$ contains an obstacle $x_{207}$ with the property that the differences $x_{457}-x_{207}=\frac{343}{4}$ and $x_{207}-x_{450}=\frac{1473}{4}$ are both greater than the difference of the elements of the reproducing pairs, namely
$$1=\min\left\{2\Delta\mathcal{M}_i:i\in\{205,\ldots,n\}\right\}=:\mu_n$$
for every $n\in\{205,\ldots,461\}$. Randomness of obstacles makes it possible that the value of this quantity decreases at a future time step, hence bringing the median sequence closer to convergence. Let us explain the first decrease in the normal form orbit of order $205$.

The first irregular phase continues with the following recursive reproductions of linear pairs of difference $1$ (figure \ref{fig:205mid}):
\begin{center}
$\left[x_{457},x_{458}\right]\mapsto\left[x_{463},x_{464}\right]\mapsto\left[x_{469},x_{470}\right]\mapsto\left[x_{475},x_{476}\right] \mapsto\left[x_{481},x_{482}\right]\mapsto\left[x_{485},x_{486}\right]\mapsto\left[x_{493},x_{494}\right]\mapsto \left[x_{497},x_{498}\right]\mapsto\left[x_{505},x_{506}\right]\mapsto\left[x_{509},x_{510}\right]\mapsto\left[x_{517},x_{518}\right] \mapsto\left[x_{521},x_{522}\right]$.
\end{center}
However, it turns out that the pair $\left[x_{521},x_{522}\right]$ does not reproduce; it never becomes ready since $x_{521}<x_{228}<x_{522}$. When the median walks across the pair $\left[x_{521},x_{228}\right]$, which has a smaller difference $x_{228}-x_{521}=\frac{3}{4}$, the value of $\mu_n$ decreases, i.e., $\mu_{526}=\frac{3}{4}$, marking the end of the first irregular structure (see figures \ref{fig:205} and \ref{fig:205mu}).

This example shows that a sufficient condition for termination of the algorithm is \textit{a decrease in the minimum value of the difference of two consecutive medians}. Since this minimum value is initially $\frac{1}{2}$, we replace the condition $\mathcal{M}_{k-1}\neq\mathcal{M}_k$ in the algorithm with $\mathcal{M}_k-\mathcal{M}_{k-1}\geqslant\frac{1}{2}$. The new algorithm reads as follows.\bigskip

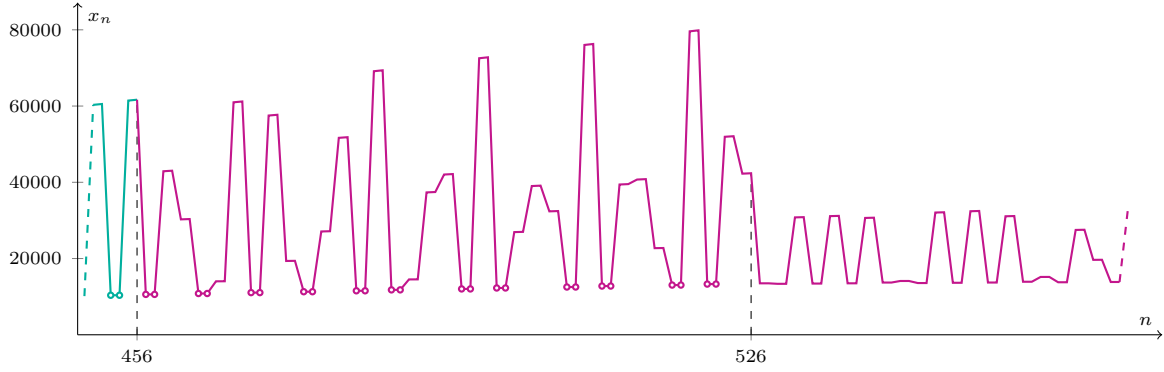
\begin{figure}
\centering
\begin{tikzpicture}
\begin{axis}[
	xmin=449.2272727,
	xmax=572.8636364,
	ymin=0,
	ymax=87120.27273,
    ytick={20000,40000,60000,80000},yticklabels={20000,40000,60000,80000},
    xtick={456,526},xticklabels={456,526},
	axis lines=middle,
	samples=100,
	xlabel=$n$,
	ylabel=$x_n$,
	width=16cm,
	height=6cm,
	clip=false,
	axis lines=middle,
    x axis line style=->,
    y axis line style=->,
    scaled y ticks=false
]
\draw[dashed] (axis cs:456,0)--(axis cs:456., 61664.750);
\draw[dashed] (axis cs:526,0)--(axis cs:526., 42371.37500);

\draw[thick,newlightblue,dashed] (axis cs:450., 10138.) -- (axis cs:451., 60312.750);

\draw[thick,color=newlightblue] plot coordinates {(axis cs:451., 60312.750) (axis cs:452., 60536.250) (axis cs:453., 10363.500) (axis cs:454., 10364.500) (axis cs:455., 61439.250) (axis cs:456., 61664.750)};

\draw[thick,color=newpurple] plot coordinates {(axis cs:456., 61664.750) (axis cs:457., 10592.) (axis cs:458., 10593.) (axis cs:459., 42896.12500) (axis cs:460., 43037.87500) (axis cs:461., 30271.62500) (axis cs:462., 30357.37500) (axis cs:463., 10823.50000) (axis cs:464., 10824.50000) (axis cs:465., 14022.37500) (axis cs:466., 14037.12500) (axis cs:467., 60985.37500) (axis cs:468., 61201.12500) (axis cs:469., 11058.) (axis cs:470., 11059.) (axis cs:471., 57512.37500) (axis cs:472., 57710.62500) (axis cs:473., 19359.37500) (axis cs:474., 19394.62500) (axis cs:475., 11295.50000) (axis cs:476., 11296.50000) (axis cs:477., 27098.12500) (axis cs:478., 27165.37500) (axis cs:479., 51661.62500) (axis cs:480., 51830.87500) (axis cs:481., 11536.) (axis cs:482., 11537.) (axis cs:483., 69135.75000) (axis cs:484., 69375.25000) (axis cs:485., 11778.50000) (axis cs:486., 11779.50000) (axis cs:487., 14519.87500) (axis cs:488., 14532.12500) (axis cs:489., 37344.) (axis cs:490., 37449.50000) (axis cs:491., 42035.37500) (axis cs:492., 42159.12500) (axis cs:493., 12025.) (axis cs:494., 12026.) (axis cs:495., 72540.75000) (axis cs:496., 72786.25000) (axis cs:497., 12273.50000) (axis cs:498., 12274.50000) (axis cs:499., 26933.62500) (axis cs:500., 26993.37500) (axis cs:501., 39014.50000) (axis cs:502., 39122.) (axis cs:503., 32376.12500) (axis cs:504., 32456.37500) (axis cs:505., 12526.) (axis cs:506., 12527.) (axis cs:507., 76029.75000) (axis cs:508., 76281.25000) (axis cs:509., 12780.50000) (axis cs:510., 12781.50000) (axis cs:511., 39418.37500) (axis cs:512., 39523.62500) (axis cs:513., 40719.) (axis cs:514., 40828.50000) (axis cs:515., 22719.87500) (axis cs:516., 22758.62500) (axis cs:517., 13039.) (axis cs:518., 13040.) (axis cs:519., 79602.75000) (axis cs:520., 79860.25000) (axis cs:521., 13299.50000) (axis cs:522., 13300.50000) (axis cs:523., 51938.12500) (axis cs:524., 52086.87500) (axis cs:525., 42260.62500) (axis cs:526., 42371.37500) (axis cs:527., 13497.12500) (axis cs:528., 13497.87500) (axis cs:529., 13366.37500) (axis cs:530., 13366.62500) (axis cs:531., 30790.31250) (axis cs:532., 30856.18750) (axis cs:533., 13433.) (axis cs:534., 13433.25000) (axis cs:535., 31121.93750) (axis cs:536., 31188.31250) (axis cs:537., 13500.12500) (axis cs:538., 13500.37500) (axis cs:539., 30647.56250) (axis cs:540., 30711.43750) (axis cs:541., 13700.) (axis cs:542., 13700.75000) (axis cs:543., 14108.75000) (axis cs:544., 14111.) (axis cs:545., 13568.25000) (axis cs:546., 13568.50000) (axis cs:547., 32064.18750) (axis cs:548., 32132.06250) (axis cs:549., 13636.87500) (axis cs:550., 13637.12500) (axis cs:551., 32405.81250) (axis cs:552., 32474.18750) (axis cs:553., 13706.) (axis cs:554., 13706.25000) (axis cs:555., 31084.93750) (axis cs:556., 31147.81250) (axis cs:557., 13908.87500) (axis cs:558., 13909.62500) (axis cs:559., 15168.12500) (axis cs:560., 15173.37500) (axis cs:561., 13776.12500) (axis cs:562., 13776.37500) (axis cs:563., 27499.75000) (axis cs:564., 27548.75000) (axis cs:565., 19652.43750) (axis cs:566., 19673.31250) (axis cs:567., 13847.) (axis cs:568., 13847.25000)};

\draw[thick,dashed,newpurple] (axis cs:568., 13847.25000) -- (axis cs:569., 33869.18750);

\draw[thick,newlightblue,fill=white] (axis cs:453., 10363.500) circle (1pt);
\draw[thick,newlightblue,fill=white] (axis cs:454., 10364.500) circle (1pt);

\draw[thick,newpurple,fill=white] (axis cs:457., 10592.) circle (1pt);
\draw[thick,newpurple,fill=white] (axis cs:458., 10593.) circle (1pt);

\draw[thick,newpurple,fill=white] (axis cs:463., 10823.50000) circle (1pt);
\draw[thick,newpurple,fill=white] (axis cs:464., 10824.50000) circle (1pt);

\draw[thick,newpurple,fill=white] (axis cs:469., 11058.) circle (1pt);
\draw[thick,newpurple,fill=white] (axis cs:470., 11059.) circle (1pt);

\draw[thick,newpurple,fill=white] (axis cs:475., 11295.50000) circle (1pt);
\draw[thick,newpurple,fill=white] (axis cs:476., 11296.50000) circle (1pt);

\draw[thick,newpurple,fill=white] (axis cs:481., 11536.) circle (1pt);
\draw[thick,newpurple,fill=white] (axis cs:482., 11537.) circle (1pt);

\draw[thick,newpurple,fill=white] (axis cs:485., 11778.50000) circle (1pt);
\draw[thick,newpurple,fill=white] (axis cs:486., 11779.50000) circle (1pt);

\draw[thick,newpurple,fill=white] (axis cs:493., 12025.) circle (1pt);
\draw[thick,newpurple,fill=white] (axis cs:494., 12026.) circle (1pt);

\draw[thick,newpurple,fill=white] (axis cs:497., 12273.50000) circle (1pt);
\draw[thick,newpurple,fill=white] (axis cs:498., 12274.50000) circle (1pt);

\draw[thick,newpurple,fill=white] (axis cs:505., 12526.) circle (1pt);
\draw[thick,newpurple,fill=white] (axis cs:506., 12527.) circle (1pt);

\draw[thick,newpurple,fill=white] (axis cs:509., 12780.50000) circle (1pt);
\draw[thick,newpurple,fill=white] (axis cs:510., 12781.50000) circle (1pt);

\draw[thick,newpurple,fill=white] (axis cs:517., 13039.) circle (1pt);
\draw[thick,newpurple,fill=white] (axis cs:518., 13040.) circle (1pt);

\draw[thick,newpurple,fill=white] (axis cs:521., 13299.50000) circle (1pt);
\draw[thick,newpurple,fill=white] (axis cs:522., 13300.50000) circle (1pt);
\end{axis}
\end{tikzpicture}
\caption{\label{fig:205mid}\small The first irregular phase of the normal form orbit of order $205$.}
\end{figure}

%%%%%%%%%%%%%%%%%%%%%%%%%%%%%%%%%%%%%%%%%%%%%%%%%%%%%%%% FIGURE
\begin{figure}
\centering
\input{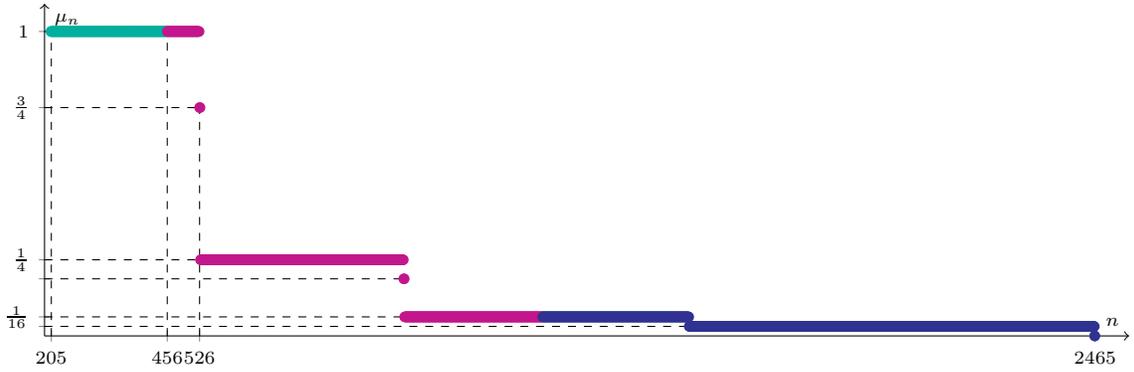}
\caption{\label{fig:205mu}\small The sequence $\left(\mu_n\right)_{n=205}^{2465}$.}
\end{figure}
%%%%%%%%%%%%%%%%%%%%%%%%%%%%%%%%%%%%%%%%%%%%%%%%%%%%%%%%

\begin{algorithmb}[Recursive reproductions of pairs]\textcolor{white}{a}
\begin{itemize}[leftmargin=2cm,itemsep=0pt]
\item[{\sc input:}] A set $\xi_{n_0}$ with R-subset $\PP_0=[0,1]$, where $n_{0}\geqslant 3$ is odd and $x_{n_0+1}\geqslant\max\left(\PP_0\right)$.
\item[{\sc output:}] Sequences $\left(n_0,\ldots,n_{N-1}\right)$ and
$\left(\PP_1,\ldots,\PP_N\right)=\left(\left[x_{n_i+2},x_{n_i+3}\right]\right)_{i=0}^{N-1}$,
where
\begin{enumerate}
\item[i)] For every $i\in\{0,\ldots,N-1\}$, $\PP_i\subseteq \xi_{n_i}$ is ready and generates
$\PP_{i+1}$ according to part ii) of corollary \ref{cor:firstdifferenceaffine}.
\item[ii)] $N$ is the largest positive integer $i$ such that $\PP_i$ exists and $\min\{2\Delta\mathcal{M}_n:n_0+1\leqslant n\leqslant n_{i-1}\}=1$.
\end{enumerate}
\end{itemize}
\begin{enumerate}[font=\scriptsize]
\item[\textnormal{1}\,\,\,\,] $\textnormal{go-on}:=\textit{true}$
\item[\textnormal{2}\,\,\,\,] $i:=0$
\item[\textnormal{3}\,\,\,\,] \textnormal{\textbf{while $\textnormal{go-on}$ do}}
\item[\textnormal{4}\,\,\,\,] \hspace*{0.4cm} $x_{n_i+2}:=i+1+\frac{1}{2}\sum_{\ell=0}^i n_\ell$
\item[\textnormal{5}\,\,\,\,] \hspace*{0.4cm} $x_{n_i+3}:=x_{n_i+2}+1$
\item[\textnormal{6}\,\,\,\,] \hspace*{0.4cm} $\PP_{i+1}:=\left[x_{n_i+2},x_{n_i+3}\right]$
\item[\textnormal{7}\,\,\,\,] \hspace*{0.4cm} $k:=n_{i}+2$
\item[\textnormal{8}\,\,\,\,] \hspace*{0.4cm} $\textnormal{go-on}:=\textit{false}$
\item[\textnormal{9}\,\,\,\,] \hspace*{0.4cm} \textnormal{\textbf{while $\Delta\mathcal{M}_{k}\geqslant\frac{1}{2}$ and $\mathcal{M}_k\leqslant x_{n_i+2}$ do}}
\item[\textnormal{10}\,\,\,\,] \hspace*{0.4cm}\hspace*{0.4cm} \textnormal{\textbf{if $\PP_{i+1}\subseteq\xi_k$\textnormal{ is ready} and $x_{k+1}\geqslant\max\left(\PP_{i+1}\right)$ then}}
\item[\textnormal{11}\,\,\,\,] \hspace*{0.4cm}\hspace*{0.4cm} \hspace*{0.4cm} $n_{i+1}:=k$
\item[\textnormal{12}\,\,\,\,] \hspace*{0.4cm}\hspace*{0.4cm} \hspace*{0.4cm} $\textnormal{go-on}:=\textit{true}$
\item[\textnormal{13}\,\,\,\,] \hspace*{0.4cm}\hspace*{0.4cm} \textnormal{\textbf{end if}}
\item[\textnormal{14}\,\,\,\,] \hspace*{0.4cm}\hspace*{0.4cm} $k:=k+2$
\item[\textnormal{15}\,\,\,\,] \hspace*{0.4cm} \textnormal{\textbf{end do}}
\item[\textnormal{16}\,\,\,\,] \hspace*{0.4cm} $i:=i+1$
\item[\textnormal{17}\,\,\,\,] \textnormal{\textbf{end do}}
\item[\textnormal{18}\,\,\,\,] \textnormal{\textbf{return} $\left(\left(n_0,\ldots,n_{N-1}\right),\left(\PP_1,\ldots,\PP_N\right)\right)$}
\end{enumerate}
\end{algorithmb}\bigskip

\noindent Clearly, we have the following.\bigskip

\begin{proposition}
Algorithm B is correct and terminates if there exists $k\geqslant n_0+1$ such that $\Delta\mathcal{M}_k<\frac{1}{2}$.
\end{proposition}\bigskip

%==============================================================================================================

\section{Proof of main theorem}\label{section:Proof}

\noindent Let us first sketch the idea of the proof.

Choose a quasi-regular phase; for instance, that in which there are 
exactly four reproducing arithmetic progressions $\AP_0=[0,1,2]$, $\AP_1$, $\AP_2$, $\AP_3$ and the gap between every two consecutive progressions is empty. 
Using algorithm A, we then construct an initial set whose orbit begins with 
this phase, by identifying sufficient conditions on 
its non-participating elements ---in the form of inequalities--- 
which guarantee the absence of obstacles until the median sequence reaches the largest element of $\AP_2$ (therefore completing the production of $\AP_3$). 
Moreover, by positioning a non-central element of the initial set at this median, we can also force the orbit to stabilise immediately after the map finishes generating the specified structure.

The above idea can also be applied, using algorithm B, if we specify any 
form of irregular or regular phase. 
\bigskip

\begin{proofoftheoremI}
Fix $k\geqslant 0$, and fix $N\geqslant 2$. We will construct an initial set $\xi_{n_0}$ whose orbit consists of $N\geqslant 2$ reproducing pairs and stabilises at the larger element of the second-to-last pair, where the gap between every two consecutive pairs which the median walks across before stabilisation is partitioned uniformly by exactly $k$ equally-spaced obstacles which are elements of $\xi_{n_0}$.

Let us first determine the minimum value of $n_0$ in terms of $N$ and $k$. Clearly, the largest $\frac{n_0+1}{2}$ elements of $\xi_{n_0}$ must contain the first pair, $k(N-1)$ obstacles, and possibly other non-participating elements, so we must have
$$\frac{n_0+1}{2}\geqslant k(N-1)+2,$$
i.e.,
\begin{equation}\label{eq:secondconstruction}
n_0\geqslant 2k(N-1)+3\geqslant 3.
\end{equation}
The first two pairs in the orbit are $\PP_0=\left[x_{n_{-1}+2},x_{n_{-1}+3}\right]=[0,1]$ and $\PP_1=\left[x_{n_0+2},x_{n_0+3}\right]$. Since $1=\mathcal{M}_{n_0+2}$ and there are exactly $k$ obstacles between these two pairs, then $x_{n_0+2}=\mathcal{M}_{n_0+2k+4}$, which means the pair $\PP_1$ is ready at time $n_1=n_0+2k+4$. Continuing this, we have that
$$
n_i = n_{i-1}+2k+4\qquad\text{for every }i\in\{1,\ldots,N-1\}.
$$
Solving this recursion gives
$$n_i = n_{0}+(2k+4)i\qquad\text{for every }i\in\{1,\ldots,N-1\}.$$
By algorithm B, the $N$ generated pairs are
\begin{eqnarray}\label{eq:pairsthmA}
\left(\PP_1,\ldots,\PP_N\right)&=&\left(\left[x_{n_i+2},x_{n_i+3}\right]\right)_{i=0}^{N-1}\nonumber\\
 &=&\left(\left[\frac{i+1}{2}n_0+\frac{k+2}{2}i^2+\frac{k+4}{2}i+1,\frac{i+1}{2}n_0+\frac{k+2}{2}i^2+\frac{k+4}{2}i+2\right]\right)_{i=0}^{N-1}.
\end{eqnarray}

Notice that the obstacles form $N-1$ arithmetic progressions, each of length $k$, which, together with the endpoints of the respective gaps form arithmetic progressions of length $k+2$. For every $i\in\{0,\ldots,N-2\}$, the $i$-th arithmetic progression $\left[\max\left(\PP_i\right),\ldots,\min\left(\PP_{i+1}\right)\right]$ is ready at time $n_{i}+2$ and, since $x_{n_{i}+3}> \min\left(\PP_{i+1}\right)=x_{n_i+2}$, generates the intermediate iterates $x_{n_i+4},\ldots,x_{n_{i+1}+1}$ with $x_{n_i+4}<\cdots< x_{n_{i+1}+1}$ by part i) of corollary \ref{cor:firstdifferenceaffine}. Now, for every $i\in\{1,\ldots,N-2\}$ we have
\begin{eqnarray*}
x_{n_i+4}-x_{n_i+1}&=&\left[\left(n_i+4\right)\mathcal{M}_{n_i+3}-\left(n_i+3\right)\mathcal{M}_{n_i+2}\right]-\left[\left(n_i+1\right)\mathcal{M}_{n_i}-n_i\mathcal{M}_{n_i-1}\right]\\
                   &=&\left(n_i+4\right)\left\langle x_{n_{i-1}+3},x_{n_{i-1}+3}+\frac{1}{k+1}\left(x_{n_i+2}-x_{n_{i-1}+3}\right)\right\rangle-\left(n_i+3\right)x_{n_{i-1}+3}\\
                   &&-\left[\left(n_i+1\right)x_{n_{i-1}+2}-n_i\left\langle x_{n_{i-2}+3}+\frac{k}{k+1}\left(x_{n_{i-1}+2}-x_{n_{i-2}+3}\right),x_{n_{i-1}+2}\right\rangle\right]\\
                   &=&\frac{\left(-n_i+2k-2\right)x_{n_{i-1}+3}-\left(n_i+4\right)x_{n_i+2}+n_ix_{n_{i-2}+3}-\left(n_i+2k+2\right)x_{n_{i-1}+2}}{2k+2}\\
                   &=&\frac{{n_0}^2+\left[(2i+4)k+4i+10\right]n_0+\left[4ik^2+(24i+4)k+32i+4\right]}{4k+4}\,\,\,>\,\,\,0.
\end{eqnarray*}

\sloppy Therefore, to guarantee that intermediate iterates are all greater than or equal to $\max\left(\PP_{N-1}\right)=x_{n_{N-2}+3}$, it suffices to impose the condition $$x_{n_0+4}\geqslant\max\left(\PP_{N-1}\right).$$
Using \eqref{eq:pairsthmA} and the fact that
\begin{eqnarray*}
x_{n_0+4}&=&\left(n_0+4\right)\mathcal{M}_{n_0+3}-\left(n_0+3\right)\mathcal{M}_{n_0+2}\\
&=&\left(n_0+4\right)\left\langle 1,1+\frac{1}{k+1}\left(x_{n_0+2}-1\right)\right\rangle-\left(n_0+3\right)\cdot1\\
&=&\frac{{n_0}^2}{4k+4}+\frac{n_0}{k+1}+1,
\end{eqnarray*}
the condition above is equivalent to
$$\frac{{n_0}^2}{4k+4}+\frac{n_0}{k+1}+1\geqslant\frac{N-1}{2}n_0+\frac{k+2}{2}(N-2)^2+\frac{k+4}{2}(N-2)+2,$$
namely,
\begin{equation}\label{eq:polynomialthmA}
\frac{{n_0}^2}{4k+4}+\frac{(1-N)k-N+3}{2k+2}n_0-\frac{k+2}{2}N^2+\frac{3k+4}{2}N-k-1\geqslant0,
\end{equation}
or, equivalently,
\begin{equation}\label{eq:secondconstruction2}
n_0\geqslant \mathcal{N}(k,N),
\end{equation}
where
\begin{eqnarray*}
\mathcal{N}(k,N)&:=&(k+1)N-k-3\\
&&+\,\,\sqrt{\left(3k^2+8k+5\right)N^2+\left(-8k^2-22k-14\right)N+\left(5k^2+14k+13\right)}
\end{eqnarray*}
is the larger root of the quadratic polynomial in $n_0$ on the left hand side of \eqref{eq:polynomialthmA}. Using \eqref{eq:secondconstruction} and \eqref{eq:secondconstruction2}, we can now assign to $n_0$ its minimum value, namely,
$$n_0=\min\left\{n\geqslant 3\text{ odd}:n\geqslant 2k(N-1)+3\text{ and }n\geqslant \mathcal{N}(k,N)\right\}.$$
Since the second condition is stronger, we have 
\begin{equation}\label{eq:n0thmA}
n_0=
\begin{cases}
\left\lceil \mathcal{N}(k,N)\right\rceil&\text{if }\left\lceil \mathcal{N}(k,N)\right\rceil\text{ is odd}\\
\left\lceil \mathcal{N}(k,N)\right\rceil+1&\text{otherwise.}
\end{cases}
\end{equation}

Therefore, if we let
\begin{equation}\label{eq:setthmA}
\xi_{n_0}=\left[\ell_1,\ldots,\ell_{\frac{n_0-1}{2}},0,1,o_1,\ldots,o_{k(N-1)},u_1,\ldots,u_{\frac{n_0-2k(N-1)-3}{2}}\right],
\end{equation}
where $o_1,\ldots,o_{k(N-1)}$ are the specified obstacles and $\ell_1,\ldots,\ell_{\frac{n_0-1}{2}},u_1,\ldots,u_{\frac{n_0-2k(N-1)-3}{2}}$ are chosen to satisfy
$$\ell_1,\ldots,\ell_{\frac{n_0-1}{2}},u_1,\ldots,u_{\frac{n_0-2k(N-1)-3}{2}}\notin \left(0,\max\left(\PP_{N-1}\right)\right)$$
and
$$x_{n_0+1}=\max\left(\PP_{N-1}\right),\quad\text{i.e.,}\quad -\mathcal{S}\left(\xi_{n_0}\right)= \frac{N-1}{2}n_0+\frac{k+2}{2}(N-2)^2+\frac{k+4}{2}(N-2)+2,$$
then we have $\mathcal{M}_{n_{N-1}+2}=\mathcal{M}_{n_{N-1}+3}$ as the first two equal consecutive medians, implying that
\begin{eqnarray*}
\tau\left(\xi_{n_0}\right)&=&n_{N-1}+4\\
&=&n_0+(2k+4)(N-1)+4\\
&\sim& \left(1+\frac{2k+4}{k+1+\sqrt{3k^2+8k+5}}\right)n_0\\
&=&\sqrt{\frac{3k+5}{k+1}}n_0,
\end{eqnarray*}
proving the theorem.
\end{proofoftheoremI}\bigskip

Notice that the case $k=0$ of this theorem represents the regular phase of a normal form orbit, whose length is asymptotic to $\sqrt{5}$ times its order. As $k\to\infty$, this coefficient decreases and approaches $\sqrt{3}$.

As an example, let us write down an initial set prescribed by the above proof whose orbit has $k=1$ obstacle in each gap (so that the transit time of the set is asymptotic to twice its size) and $N=5$ generated pairs. First, by \eqref{eq:n0thmA}, we obtain $n_0=21$. Next, by \eqref{eq:pairsthmA}, the reproducing pairs are
$$\left(\PP_0,\PP_1,\PP_2,\PP_3,\PP_4,\PP_5\right)=\left(\left[0,1\right],\left[\frac{23}{2},\frac{25}{2}\right],\left[26,27\right],\left[\frac{87}{2},\frac{89}{2}\right],\left[64,65\right],\left[\frac{175}{2},\frac{177}{2}\right]\right),$$
so that the obstacles are
$$\left(o_1,o_2,o_3,o_4\right)=\left(\left\langle1,\frac{23}{2}\right\rangle,\left\langle\frac{25}{2},26\right\rangle,\left\langle27,\frac{87}{2}\right\rangle,\left\langle\frac{89}{2},64\right\rangle\right)=\left(\frac{25}{4},\frac{77}{4},\frac{141}{4},\frac{217}{4}\right).$$
Therefore, by \eqref{eq:setthmA}, our initial set is
$$\xi_{21}=\left[\ell_1,\ldots,\ell_{10},0,1,\frac{25}{4},\frac{77}{4},\frac{141}{4},\frac{217}{4},u_1,\ldots,u_5\right],$$
where $\ell_1,\ldots,\ell_{10},u_1,\ldots,u_5$ must be chosen to satisfy the conditions
$$\ell_1,\ldots,\ell_{10},u_1,\ldots,u_5\notin(0,65)\qquad\text{and}\qquad-\mathcal{S}\left(\xi_{21}\right)=65.$$
Let us choose $u_1=\cdots=u_5=65$ and $\ell_2=\cdots=\ell_{10}=0$, so that the second condition uniquely determines the value of
$$\ell_1=-65-\left(9\cdot0+0+1+\frac{25}{4}+\frac{77}{4}+\frac{141}{4}+\frac{217}{4}+5\cdot65\right)=-506$$
which satisfies the first condition, giving
\begin{equation}\label{eq:examplethmA}
\xi_{21}=\left[-506,\underbrace{0,\ldots,0}_{9},0,1,\frac{25}{4},\frac{77}{4},\frac{141}{4},\frac{217}{4},\underbrace{65,\ldots,65}_5\right].
\end{equation}
The transit time of this set is $49$ and its orbit is shown in figure \ref{fig:pair}.

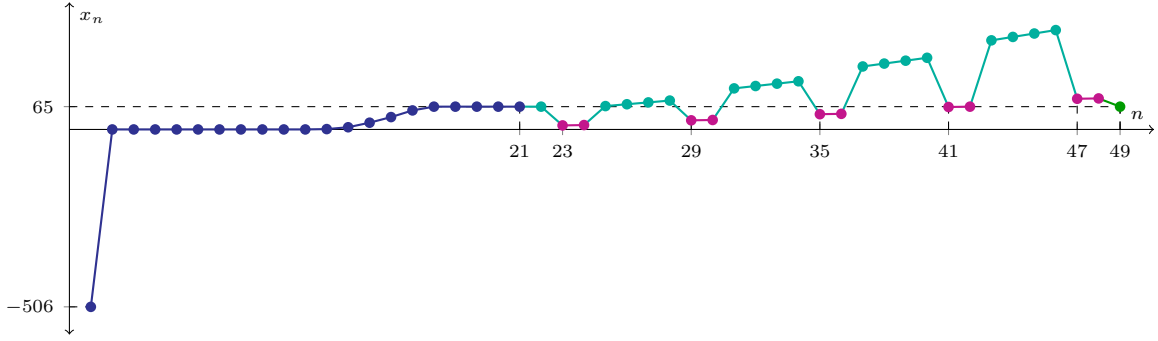
\begin{figure}
\centering
\begin{tikzpicture}
\begin{axis}[
	xmin=0,
	xmax=50.58064516,
	ymin=-584.9375000,
	ymax=362.3125000,
    xtick={21,23,29,35,41,47,49},
    ytick={-506,65},
	axis lines=middle,
    x axis line style=->,
    y axis line style=<->,
	samples=100,
	xlabel=$n$,
	ylabel=$x_n$,
	width=16cm,
    height=6cm
]

\draw[dashed] (axis cs:49,0) -- (axis cs:49,65);
\draw[dashed] (axis cs:0,65) -- (axis cs:49,65);

\draw[dashed] (axis cs:0,-506) -- (axis cs:1,-506);

\draw[dashed] (axis cs:21,0) -- (axis cs:21,65);
\draw[dashed] (axis cs:23,0) -- (axis cs:23,11.5);
\draw[dashed] (axis cs:29,0) -- (axis cs:29,26);
\draw[dashed] (axis cs:35,0) -- (axis cs:35,43.5);
\draw[dashed] (axis cs:41,0) -- (axis cs:41,64);
\draw[dashed] (axis cs:47,0) -- (axis cs:47,87.5);
\draw[dashed] (axis cs:49,0) -- (axis cs:49,65);

\draw[thick,color=newblue] plot coordinates {(axis cs:1., -506.) (axis cs:2., 0.) (axis cs:3., 0.) (axis cs:4., 0.) (axis cs:5., 0.) (axis cs:6., 0.) (axis cs:7., 0.) (axis cs:8., 0.) (axis cs:9., 0.) (axis cs:10., 0.) (axis cs:11., 0.) (axis cs:12., 1.) (axis cs:13., 6.2500000) (axis cs:14., 19.250000) (axis cs:15., 35.250000) (axis cs:16., 54.250000) (axis cs:17., 65.) (axis cs:18., 65.) (axis cs:19., 65.) (axis cs:20., 65.) (axis cs:21., 65.)};

\draw[thick,color=newlightblue] plot coordinates {(axis cs:21., 65.) (axis cs:22., 65.) (axis cs:23., 11.500000)};

\draw[thick,color=newpurple] plot coordinates {(axis cs:23., 11.500000) (axis cs:24., 12.500000)};

\draw[thick,color=newlightblue] plot coordinates {(axis cs:24., 12.500000) (axis cs:25., 66.625000) (axis cs:26., 71.875000) (axis cs:27., 77.125000) (axis cs:28., 82.375000) (axis cs:29., 26.)};

\draw[thick,color=newpurple] plot coordinates {(axis cs:29., 26.) (axis cs:30., 27.)};

\draw[thick,color=newlightblue] plot coordinates {(axis cs:30., 27.) (axis cs:31., 117.12500) (axis cs:32., 123.87500) (axis cs:33., 130.62500) (axis cs:34., 137.37500) (axis cs:35., 43.500000)};

\draw[thick,color=newpurple] plot coordinates {(axis cs:35., 43.500000) (axis cs:36., 44.500000)};

\draw[thick,color=newlightblue] plot coordinates {(axis cs:36., 44.500000) (axis cs:37., 179.62500) (axis cs:38., 187.87500) (axis cs:39., 196.12500) (axis cs:40., 204.37500) (axis cs:41., 64.)};

\draw[thick,color=newpurple] plot coordinates {(axis cs:41., 64.) (axis cs:42., 65.)};

\draw[thick,color=newlightblue] plot coordinates {(axis cs:42., 65.) (axis cs:43., 254.12500) (axis cs:44., 263.87500) (axis cs:45., 273.62500) (axis cs:46., 283.37500) (axis cs:47., 87.500000)};

\draw[thick,color=newpurple] plot coordinates {(axis cs:47., 87.500000) (axis cs:48., 88.500000)};

\draw[thick,color=newgreen] plot coordinates {(axis cs:48., 88.500000) (axis cs:49., 65.)};

\fill[newblue] (axis cs:1, -506) circle (2pt);
\fill[newblue] (axis cs:2, 0) circle (2pt);
\fill[newblue] (axis cs:3, 0) circle (2pt);
\fill[newblue] (axis cs:4, 0) circle (2pt);
\fill[newblue] (axis cs:5, 0) circle (2pt);
\fill[newblue] (axis cs:6, 0)  circle (2pt);
\fill[newblue] (axis cs:7, 0) circle (2pt);
\fill[newblue] (axis cs:8, 0) circle (2pt);
\fill[newblue] (axis cs:9, 0) circle (2pt);
\fill[newblue] (axis cs:10, 0) circle (2pt);
\fill[newblue] (axis cs:11, 0) circle (2pt);
\fill[newblue] (axis cs:12, 1) circle (2pt);
\fill[newblue] (axis cs:13, 6.25) circle (2pt);
\fill[newblue] (axis cs:14, 19.25) circle (2pt);
\fill[newblue] (axis cs:15, 35.25) circle (2pt);
\fill[newblue] (axis cs:16, 54.25) circle (2pt);
\fill[newblue] (axis cs:17, 65) circle (2pt);
\fill[newblue] (axis cs:18, 65) circle (2pt);
\fill[newblue] (axis cs:19, 65) circle (2pt);
\fill[newblue] (axis cs:20, 65) circle (2pt);
\fill[newblue] (axis cs:21, 65) circle (2pt);

\fill[newlightblue] (axis cs:22,65) circle (2pt);

\fill[newpurple] (axis cs:23., 11.500000) circle (2pt);
\fill[newpurple] (axis cs:24., 12.500000) circle (2pt);

\fill[newlightblue] (axis cs:25., 66.625000) circle (2pt);
\fill[newlightblue] (axis cs:26., 71.875000) circle (2pt);
\fill[newlightblue] (axis cs:27., 77.125000) circle (2pt);
\fill[newlightblue] (axis cs:28., 82.375000) circle (2pt);

\fill[newpurple] (axis cs:29., 26.) circle (2pt);
\fill[newpurple] (axis cs:30., 27.) circle (2pt);

\fill[newlightblue] (axis cs:31., 117.12500) circle (2pt);
\fill[newlightblue] (axis cs:32., 123.87500) circle (2pt);
\fill[newlightblue] (axis cs:33., 130.62500) circle (2pt);
\fill[newlightblue] (axis cs:34., 137.37500) circle (2pt);

\fill[newpurple] (axis cs:35., 43.500000) circle (2pt);
\fill[newpurple] (axis cs:36., 44.500000) circle (2pt);

\fill[newlightblue] (axis cs:37., 179.62500) circle (2pt);
\fill[newlightblue] (axis cs:38., 187.87500) circle (2pt);
\fill[newlightblue] (axis cs:39., 196.12500) circle (2pt);
\fill[newlightblue] (axis cs:40., 204.37500) circle (2pt);

\fill[newpurple] (axis cs:41., 64.) circle (2pt);
\fill[newpurple] (axis cs:42., 65.) circle (2pt);

\fill[newlightblue] (axis cs:43., 254.12500) circle (2pt);
\fill[newlightblue] (axis cs:44., 263.87500) circle (2pt);
\fill[newlightblue] (axis cs:45., 273.62500) circle (2pt);
\fill[newlightblue] (axis cs:46., 283.37500) circle (2pt);
   
\fill[newpurple] (axis cs:47., 87.500000) circle (2pt);
\fill[newpurple] (axis cs:48., 88.500000) circle (2pt);

\fill[newgreen] (axis cs:49., 65.) circle (2pt);
\end{axis}
\end{tikzpicture}
\caption{\label{fig:pair}\small
The orbit of the input set \eqref{eq:examplethmA}. The input set is shown in dark blue, the generated pairs in purple, the intermediate iterates in light blue, and the iterate from which the orbit stabilises in green.}
\end{figure}

Finally, let us prove part ii) of the theorem using the same idea and algorithm A. Here we choose the simplest form of quasi-regular phase, i.e., that in which there are no obstacles.\bigskip

\begin{proofoftheoremII}
Fix $N\geqslant 2$. We will construct an initial set $\xi_{n_0}$ whose orbit consists of $N\geqslant 2$ reproducing arithmetic progressions and stabilises at the largest element of the second-to-last progression, where the gap between every two consecutive progressions is empty.

Let us first determine the minimum value of $n_0$ in terms of $N$. For every $i\in\{1,\ldots,N-1\}$, there is no obstacle between $\AP_{i-1}$ and $\AP_{i}$. Since $\max\left(\AP_{i-1}\right)=\mathcal{M}_{n_{i-1}+2\left|\AP_{i-1}\right|-2}$ and there is no obstacle between these two progressions, then $\min\left(\AP_{i}\right)=\mathcal{M}_{n_{i-1}+2\left|\AP_i\right|}$, which means the progression $\AP_{i}$ is ready at time 
$$n_i=n_{i-1}+2\left|\AP_{i-1}\right|.$$
Solving this recursion gives
$$n_i=n_0+2^{i+1}+4i-2\qquad\text{for every }i\in\{1,\ldots,N-1\}.$$
The intermediate iterates are
$$x_{n_0+1},\,\,\,x_{n_1},x_{n_1+1},\,\,\,x_{n_2},x_{n_2+1},\,\,\,\ldots,\,\,\,x_{n_{N-1}},x_{n_{N-1}+1}.$$
For every $i\in\{1,\ldots,N-1\}$, by part ii) of corollary \ref{cor:firstdifferenceaffine} we have $x_{n_i+1}>x_{n_i}$ because these two iterates are generated when the median walks across the pair $\left[\max\left(\AP_{i-1}\right),\min\left(\AP_i\right)\right]$, and by algorithm A we have
\begin{eqnarray}\label{eq:firstconstruction}
x_{n_i}&=&
\begin{cases}
n_1\left\langle 2,x_{n_0+2}\right\rangle-\left(n_1-1\right)\cdot 2&\text{if }i=1\\
n_i\left\langle x_{n_{i-2}+2\left|\AP_{i-2}\right|-1},x_{n_{i-1}+2}\right\rangle-\left(n_i-1\right)x_{n_{i-2}+2\left|\AP_{i-2}\right|-1}&\text{if }2\leqslant i\leqslant N-1
\end{cases}\nonumber\\
&=&\frac{{n_0}^2}{4}+\left(\frac{5}{2}i-\frac{5}{2}+2^{i-1}\right)n_0+(i-1)2^{i+1}+5i^2-11i+5.
\end{eqnarray}
Moreover, we have
$$x_{n_i}-x_{n_{i-1}}=\left(2^{i-2}+\frac{5}{2}\right)n_0+\left(i2^i+10i-16\right)>0\qquad\text{for every }i\in\{2,\ldots,N-1\}.$$

Therefore, to guarantee that intermediate iterates are all greater than or equal to
\begin{equation}\label{eq:limitthmB}
\max\left(\AP_{N-1}\right)=\frac{N-1}{2}n_0+2^N+N^2-3N+2,
\end{equation}
it suffices to impose the condition
$$x_{n_1}\geqslant\max\left(\AP_{N-1}\right),$$
which is equivalent to
\begin{equation}\label{eq:polynomialthmB}
\frac{{n_0}^2}{4}+\frac{3-N}{2}n_0-2^{N}-N^2+3N-3\geqslant0
\end{equation}
by \eqref{eq:firstconstruction} and \eqref{eq:limitthmB}. In other words,
$$n_0\geqslant \mathcal{N}(N),$$
where
$$\mathcal{N}(N):=N-3+\sqrt{2^{N+2}+5N^2-18N+21}$$
is the larger root of the quadratic polynomial in $n_0$ on the left hand side of \eqref{eq:polynomialthmB}. Therefore, we can now assign to $n_0$ its minimum value, namely,
\begin{equation}\label{eq:n0thmB}
n_0=\min\left\{n\geqslant 5\text{ odd}:n\geqslant \mathcal{N}(N)\right\}=\begin{cases}
\left\lceil \mathcal{N}(N)\right\rceil&\text{if }\left\lceil \mathcal{N}(N)\right\rceil\text{ is odd}\\
\left\lceil \mathcal{N}(N)\right\rceil+1&\text{otherwise.}
\end{cases}
\end{equation}

Therefore, if we let
\begin{equation}\label{eq:setthmB}
\xi_{n_0}=\left[\ell_1,\ldots,\ell_{\frac{n_0-1}{2}},0,1,2,u_1,\ldots,u_{\frac{n_0-5}{2}}\right],
\end{equation}
where $\ell_1,\ldots,\ell_{\frac{n_0-1}{2}},u_1,\ldots,u_{\frac{n_0-5}{2}}$ are chosen to satisfy
$$\ell_1,\ldots,\ell_{\frac{n_0-1}{2}},u_1,\ldots,u_{\frac{n_0-5}{2}}\notin \left(0,\max\left(\AP_{N-1}\right)\right)$$
and
$$x_{n_0+1}=\max\left(\AP_{N-1}\right),\quad\text{i.e.,}\quad -\mathcal{S}\left(\xi_{n_0}\right)= \frac{N-1}{2}n_0+2^N+N^2-3N+2,$$
then we have $\mathcal{M}_{n_{N-1}+\left|\AP_N\right|}=\mathcal{M}_{n_{N-1}+\left|\AP_N\right|+1}$ as the first two equal consecutive medians, implying that
\begin{eqnarray*}
\tau\left(\xi_{n_0}\right)&=&n_{N-1}+\left|\AP_N\right|+2\\
&=&n_0+2^{N+1}+4N-2\\
&\sim& n_0+\frac{\left(2^{\frac{N+2}{2}}\right)^2}{2}\\
&\sim& n_0+\frac{{n_0}^2}{2}\\
&\sim& \frac{{n_0}^2}{2},
\end{eqnarray*}
proving the theorem.
\end{proofoftheoremII}\bigskip

For example, an initial set prescribed by the above proof whose orbit contains $N=4$ generated arithmetic progressions has size $n_0=11$, by \eqref{eq:n0thmB}. By \eqref{eq:limitthmB}, we obtain $\max\left(\AP_3\right)=\frac{77}{2}$. Therefore, by \eqref{eq:setthmA}, such an initial set is of the form
$$\xi_{11}=\left[\ell_1,\ell_2,\ell_3,\ell_4,\ell_{5},0,1,2,u_1,u_2,u_3\right],$$
where $\ell_1,\ell_2,\ell_3,\ell_4,\ell_{5},u_1,u_2,u_3$ must be chosen to satisfy the conditions
$$\ell_1,\ell_2,\ell_3,\ell_4,\ell_{5},u_1,u_2,u_3\notin\left(0,\frac{77}{2}\right)\qquad\text{and}\qquad-\mathcal{S}\left(\xi_{11}\right)=\frac{77}{2}.$$
Choosing $u_1=u_2=u_3=\frac{77}{2}$ and $\ell_2=\ell_3=\ell_4=\ell_{5}=0$, we obtain from the second condition that
$$\ell_1=-\frac{77}{2}-\left(4\cdot0+0+1+2+3\cdot\frac{77}{2}\right)=-157,$$
which satisfies the first condition, and hence
\begin{equation}\label{eq:examplethmB}
\xi_{11}=\left[-157, 0, 0, 0, 0, 0, 1, 2, \frac{77}{2}, \frac{77}{2}, \frac{77}{2}\right].
\end{equation}
The transit time of this set is $57$ and its orbit is shown in figure \ref{fig:AP}.

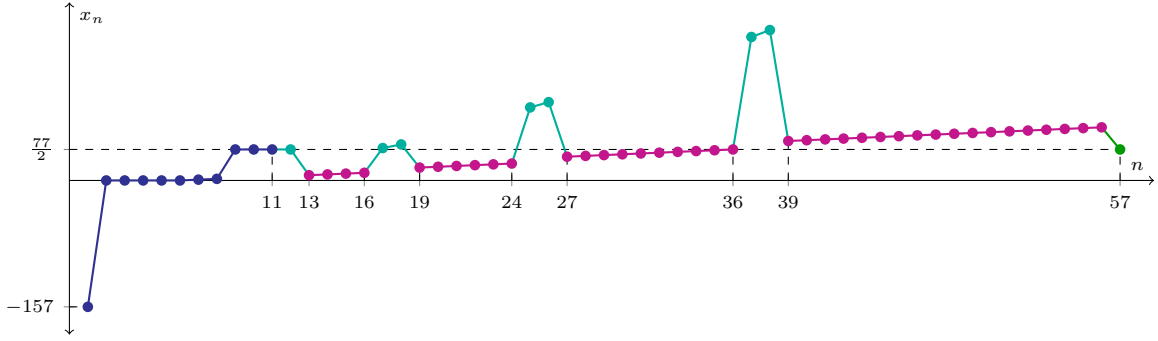
\begin{figure}
\centering
\begin{tikzpicture}
\begin{axis}[
	xmin=0,
	xmax=58.83870968,
	ymin=-191.3750000, 
	ymax=221.1250000,
    xtick={11,13,16,19,24,27,36,39,57},
    ytick={-157,38.5},
    yticklabels={$-157$,$\frac{77}{2}$},
	axis lines=middle,
    x axis line style=->,
    y axis line style=<->,
	samples=100,
	xlabel=$n$,
	ylabel=$x_n$,
	width=16cm,
    height=6cm
]

\draw[dashed] (axis cs:57,0) -- (axis cs:57,38.5);
\draw[dashed] (axis cs:0,38.5) -- (axis cs:57,38.5);

\draw[dashed] (axis cs:0,-157) -- (axis cs:1,-157);

\draw[dashed] (axis cs:11,0) -- (axis cs:11,38.5);
\draw[dashed] (axis cs:13,0) -- (axis cs:13,6.5);
\draw[dashed] (axis cs:16,0) -- (axis cs:16,9.5);
\draw[dashed] (axis cs:19,0) -- (axis cs:19,16);
\draw[dashed] (axis cs:24,0) -- (axis cs:24,21);
\draw[dashed] (axis cs:27,0) -- (axis cs:27,29.5);
\draw[dashed] (axis cs:36,0) -- (axis cs:36,38.5);
\draw[dashed] (axis cs:39,0) -- (axis cs:39,49);

\draw[thick,color=newblue] plot coordinates {(axis cs:1, -157) (axis cs:2, 0) (axis cs:3, 0) (axis cs:4, 0) (axis cs:5, 0) (axis cs:6, 0.) (axis cs:7, 1.) (axis cs:8, 2.) (axis cs:9, 38.50) (axis cs:10, 38.50) (axis cs:11, 38.50)};

\draw[thick,color=newlightblue] plot coordinates {(axis cs:11, 38.50) (axis cs:12, 38.50) (axis cs:13, 6.500)};

\draw[thick,color=newpurple] plot coordinates {(axis cs:13, 6.500) (axis cs:14, 7.500) (axis cs:15, 8.500) (axis cs:16, 9.500)};

\draw[thick,color=newlightblue] plot coordinates {(axis cs:16, 9.500) (axis cs:17, 40.25) (axis cs:18, 44.75) (axis cs:19, 16.)};

\draw[thick,color=newpurple] plot coordinates {(axis cs:19, 16.) (axis cs:20, 17.) (axis cs:21, 18.) (axis cs:22, 19.) (axis cs:23, 20.) (axis cs:24, 21.)};

\draw[thick,color=newlightblue] plot coordinates {(axis cs:24, 21.) (axis cs:25, 90.75) (axis cs:26, 97.25) (axis cs:27, 29.50)};

\draw[thick,color=newpurple] plot coordinates {(axis cs:27,29.50) (axis cs:28,30.50) (axis cs:29,31.50) (axis cs:30,32.50) (axis cs:31,33.50) (axis cs:32,34.50) (axis cs:33,35.50) (axis cs:34,36.50) (axis cs:35,37.50) (axis cs:36,38.50)};

\draw[thick,color=newlightblue] plot coordinates {(axis cs:36, 38.50) (axis cs:37, 178.2) (axis cs:38, 186.8) (axis cs:39, 49.)};

\draw[thick,color=newpurple] plot coordinates {(axis cs:39, 49.) (axis cs:40, 50.) (axis cs:41, 51.) (axis cs:42, 52.) (axis cs:43, 53.) (axis cs:44, 54.) (axis cs:45, 55.) (axis cs:46, 56.) (axis cs:47, 57.) (axis cs:48, 58.) (axis cs:49, 59.) (axis cs:50, 60.) (axis cs:51, 61.) (axis cs:52, 62.) (axis cs:53, 63.) (axis cs:54, 64.) (axis cs:55, 65.) (axis cs:56, 66.)};

\draw[thick,color=newgreen] plot coordinates {(axis cs:56, 66.) (axis cs:57, 38.50)};

\fill[newblue] (axis cs:1, -157) circle (2pt);
\fill[newblue] (axis cs:2, 0) circle (2pt);
\fill[newblue] (axis cs:3, 0) circle (2pt);
\fill[newblue] (axis cs:4, 0) circle (2pt);
\fill[newblue] (axis cs:5, 0) circle (2pt);
\fill[newblue] (axis cs:6, 0) circle (2pt);
\fill[newblue] (axis cs:7, 1) circle (2pt);
\fill[newblue] (axis cs:8, 2) circle (2pt);
\fill[newblue] (axis cs:9, 38.50) circle (2pt);
\fill[newblue] (axis cs:10, 38.50) circle (2pt);
\fill[newblue] (axis cs:11, 38.50) circle (2pt);

\fill[newlightblue] (axis cs:12, 38.50) circle (2pt);

\fill[newpurple] (axis cs:13, 6.500) circle (2pt);
\fill[newpurple] (axis cs:14, 7.500) circle (2pt);
\fill[newpurple] (axis cs:15, 8.500) circle (2pt);
\fill[newpurple] (axis cs:16, 9.500) circle (2pt);

\fill[newlightblue] (axis cs:17, 40.25) circle (2pt);
\fill[newlightblue] (axis cs:18, 44.75) circle (2pt);

\fill[newpurple] (axis cs:19, 16.) circle (2pt);
\fill[newpurple] (axis cs:20, 17.) circle (2pt);
\fill[newpurple] (axis cs:21, 18.) circle (2pt);
\fill[newpurple] (axis cs:22, 19.) circle (2pt);
\fill[newpurple] (axis cs:23, 20.) circle (2pt);
\fill[newpurple] (axis cs:24, 21.) circle (2pt);

\fill[newlightblue] (axis cs:25, 90.75) circle (2pt);
\fill[newlightblue] (axis cs:26, 97.25) circle (2pt);

\fill[newpurple] (axis cs:27,29.50) circle (2pt);
\fill[newpurple] (axis cs:28,30.50) circle (2pt);
\fill[newpurple] (axis cs:29,31.50) circle (2pt);
\fill[newpurple] (axis cs:30,32.50) circle (2pt);
\fill[newpurple] (axis cs:31,33.50) circle (2pt);
\fill[newpurple] (axis cs:32,34.50) circle (2pt);
\fill[newpurple] (axis cs:33,35.50) circle (2pt);
\fill[newpurple] (axis cs:34,36.50) circle (2pt);
\fill[newpurple] (axis cs:35,37.50) circle (2pt);
\fill[newpurple] (axis cs:36,38.50) circle (2pt);

\fill[newlightblue] (axis cs:37, 178.2) circle (2pt);
\fill[newlightblue] (axis cs:38, 186.8) circle (2pt);

\fill[newpurple] (axis cs:39, 49.) circle (2pt);
\fill[newpurple] (axis cs:40, 50.) circle (2pt);
\fill[newpurple] (axis cs:41, 51.) circle (2pt);
\fill[newpurple] (axis cs:42, 52.) circle (2pt);
\fill[newpurple] (axis cs:43, 53.) circle (2pt);
\fill[newpurple] (axis cs:44, 54.) circle (2pt);
\fill[newpurple] (axis cs:45, 55.) circle (2pt);
\fill[newpurple] (axis cs:46, 56.) circle (2pt);
\fill[newpurple] (axis cs:47, 57.) circle (2pt);
\fill[newpurple] (axis cs:48, 58.) circle (2pt);
\fill[newpurple] (axis cs:49, 59.) circle (2pt);
\fill[newpurple] (axis cs:50, 60.) circle (2pt);
\fill[newpurple] (axis cs:51, 61.) circle (2pt);
\fill[newpurple] (axis cs:52, 62.) circle (2pt);
\fill[newpurple] (axis cs:53, 63.) circle (2pt);
\fill[newpurple] (axis cs:54, 64.) circle (2pt);
\fill[newpurple] (axis cs:55, 65.) circle (2pt);
\fill[newpurple] (axis cs:56, 66.) circle (2pt);

\fill[newgreen] (axis cs:57, 38.50) circle (2pt);
\end{axis}
\end{tikzpicture}
\caption{\label{fig:AP}\small
The orbit of the input set \eqref{eq:examplethmB}. The input set is shown in dark blue, the generated pairs in purple, the intermediate iterates in light blue, and the iterate from which the orbit stabilises in green.}
\end{figure}

\section{A height-type function}

We have established our main result by considering 
sequences of initial 
sets of increasing size and expressing the growth of the transit time as 
a function of the size of the initial set. 
However, the available data for the original system $\left[0,\frac{p}{q},1\right]$ 
(figure \ref{fig:tauoriginal}) express the growth of the transit time 
as a function of the denominator $q$ of the initial condition. 
For the purpose of comparison, in the main theorem we need to replace 
the size of the initial set with a quantity which is comparable to $q$.
Accordingly, we introduce a height-type function $\mathcal{H}$ as follows.

First, we define the height of a singleton rational set $\left[\frac{p}{q}\right]$, where $q\neq 0$ and $\gcd(p,q)=1$, as
$$\mathcal{H}\left(\left[\frac{p}{q}\right]\right):=\begin{cases}
|p|&\text{if }q=1\\
|p|+|q|&\text{otherwise.}
\end{cases}$$
Then, the height of any rational set $\xi$ is defined to be
$$\mathcal{H}(\xi):=\min\biggl\{\sum_{x\in a\xi+b}\mathcal{H}([x]):a,b\in\mathbb{Q}, a\neq0\biggr\},$$
i.e., the minimum total height of the elements of a rational set which is affine-equivalent to $\xi$. With this definition, we have $\mathcal{H}\left(\left[0,\frac{p}{q},1\right]\right)=q$ for $0<p<q$, as easily verified. Hence for the original system we observe that the growths of the average and maximum transit times with the height of the initial set are algebraic with exponents $\alpha\approx 0.42$ and $\beta\approx 1.45$, respectively (figure \ref{fig:tauoriginal}).

We shall now use sets of the form \eqref{eq:setthmB} to construct a sequence of sets whose transit time grows algebraically with the height with exponent lying between $\alpha$ and $\beta$, 
thereby providing an efficient construction of initial sets with large transit times.

Given an odd integer $N\geqslant 2$, a set prescribed by the proof has the form \eqref{eq:setthmB}, where
$$n_0=\min\left\{n\geqslant 5\text{ odd}:n\geqslant N-3+\sqrt{2^{N+2}+5N^2-18N+21}\right\}\sim 2^{\frac{N}{2}+1},$$
by \eqref{eq:n0thmB}. Choosing
$$u_1=\cdots=u_{\frac{n_0-5}{2}}=m:=\frac{N-1}{2}n_0+2^N+N^2-3N+2\sim 2^N$$
and
$$\ell_2=\cdots=\ell_{\frac{n_0-1}{2}}=0$$
gives the set
$$\xi_{n_0}=\bigl[\ell_1,\underbrace{0,\ldots,0}_{\frac{n_0-3}{2}},0,1,2,\underbrace{m,\ldots,m}_{\frac{n_0-5}{2}}\bigr],$$
where
$$\ell_1=-m-3-\frac{n_0-5}{2}m\sim -2^{\frac{3}{2}N+1},$$
having transit time
\begin{equation}\label{eq:taulastsection}
\tau\left(\xi_{n_0}\right)\sim\frac{{n_0}^2}{2}\sim 2^{N+1}.
\end{equation}
Since both $N$ and $n_0$ are odd, then $m,\ell_1\in\mathbb{Z}$, so
\begin{eqnarray*}
\mathcal{H}\left(\xi_{n_0}\right)\leqslant\sum_{x\in\xi_{n_0}}\mathcal{H}([x])&=&\left|\ell_1\right|+\frac{n_0-1}{2}\cdot0+1+2+\frac{n_0-5}{2}\cdot m\\
&=&\left(m+3+\frac{n_0-5}{2}\cdot m\right)+3+\frac{n_0-5}{2}\cdot m\\
&=&\left(n_0-4\right)m+6\,\,\,=:\,\,\,\mathcal{H}^+\left(\xi_{n_0}\right).
\end{eqnarray*}
On the other hand,
$$\mathcal{H}\left(\xi_{n_0}\right)\geqslant|m|=:\mathcal{H}^-\left(\xi_{n_0}\right).$$
We have
$$\mathcal{H}^-\left(\xi_{n_0}\right)\sim 2^N\qquad\text{and}\qquad\mathcal{H}^+\left(\xi_{n_0}\right)\sim 2\cdot 2^{\frac{3}{2}N}.$$
Considering \eqref{eq:taulastsection}, we have
$$\tau\left(\xi_{n_0}\right)\sim 2\mathcal{H}^-\left(\xi_{n_0}\right)\qquad\text{and}\qquad\tau\left(\xi_{n_0}\right)\sim\sqrt[3]{2}\left[\mathcal{H}^+\left(\xi_{n_0}\right)\right]^{\frac{2}{3}}.$$
Thus, the growth of $\tau$ is exponential, the detailed behaviour depending on the value. We note that the exponent of both upper and lower bounds lies between the empirical exponents $\alpha$ and $\beta$ mentioned in the introduction.

As an illustration, suppose we want to construct an initial set with transit time at least $1000$. According to figure \ref{fig:tauoriginal} (a), if we use the original system, we expect the height of the set to be approximately $\left(\frac{1000}{e^{2.31}}\right)^{\frac{1}{0.42}}\approx 56786$. However, for $N=9$, we have from part ii) of our main theorem that the set
$$\xi_{55}=[-19703,\underbrace{0,\ldots,0}_{26},0,1,2,\underbrace{788,\ldots,788}_{25}]$$
has transit time $55+2^{9+1}+4\cdot 9-2=1113$. The height of this set is
\begin{eqnarray*}
\mathcal{H}\left(\xi_{55}\right)&=&\mathcal{H}\biggl(\biggl[-\frac{19703}{788},\underbrace{0,\ldots,0}_{26},0,\frac{1}{788},\frac{1}{394},\underbrace{1,\ldots,1}_{25}\biggr]\biggr)\\
&=&19703+788+1+788+1+394+25\cdot 1\\
&=&21700.
\end{eqnarray*}

%%%%%%%%%%%%%%%%%%%%%%%%%%%%%%%%%%%%%%%%%%%%%%%%%%%%%%%%%%%%%%%%%%%%%%%%%%%

\end{document}